\documentclass[a4paper, cleveref, autoref, thm-restate, numberwithinsect]{lipics-v2021}
\usepackage{mathtools}
\usepackage{physics}
\usepackage{graphicx}
\usepackage[dvipsnames, table]{xcolor}
\usepackage{cite}
\usepackage{bm}
\usepackage{booktabs}
\usepackage{diagbox}
\usepackage{MnSymbol}
\usepackage{complexity}
\usepackage{tabularx}
\usepackage{thm-restate}
\usepackage{xspace}
\usepackage{xargs}
\usepackage{ifthen}

\title{On \texorpdfstring{$\bm{t}$}{t}-colorable \texorpdfstring{$\bm{k}$}{k}-plane drawings}

\author{Miriam Goetze}{Karlsruhe Institute of Technology}{miriam.goetze@kit.edu}{https://orcid.org/0000-0001-8746-522X}{funded by the Deutsche Forschungsgemeinschaft (DFG, German Research Foundation) -- 520723789}
\author{Michael Kaufmann}{University of Tübingen}{michael.kaufmann@uni-tuebingen.de}{}{}
\author{Soeren Terziadis}{TU Munich and TU Eindhoven}{soeren.terziadis@ac.tuwien.ac.at}{https://orcid.org/0000-0001-5161-3841}{}

\authorrunning{M. Goetze, M. Kaufmann, S. Terziadis}

\Copyright{M. Goetze, M. Kaufmann, S. Terziadis}



\category{Track 1, Long Paper}

\acknowledgements{This research was initiated at the Workshop on \emph{Graph and Network Visualization} (GNV~2025) in Chania, Greece, June 22--27, 2025.}

\relatedversion{}

\nolinenumbers

\EventEditors{Maarten L\"{o}ffler and Silvia Miksch}
\EventNoEds{2}
\EventLongTitle{34th International Symposium on Graph Drawing and Network Visualization (GD 2026)}
\EventShortTitle{GD 2026}
\EventAcronym{GD}
\EventYear{2026}
\EventDate{August 19--21, 2026}
\EventLocation{Ontario, Canada}
\EventLogo{}
\SeriesVolume{396}
\ArticleNo{32}

\newsavebox{\imagebox}

\definecolor{lightgrayYES}{RGB}{217,217,217}
\newcommand{\shortref}[1]{{\color{eurocgbluedark}$\langle$\ref{#1}$\rangle$}}

\newcommand{\shortrefs}[2]{{\color{eurocgbluedark}$\langle$\ref{#1},\ref{#2}$\rangle$}}

\newcommand{\ra}[1]{\renewcommand{\arraystretch}{#1}}
\newcolumntype{Y}{@{\extracolsep{3pt}}>{\centering\arraybackslash}X@{\extracolsep{0pt}}}

\newcommand{\FVtII}{\raisebox{-4.5pt}{\includegraphics[page=1]{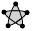}}}
\newcommand{\FVItII}{\raisebox{-4.5pt}{\includegraphics[page=3]{figures/small-icons_no_white_border.pdf}}}
\newcommand{\FVItIII}{\raisebox{-4.5pt}{\includegraphics[page=5]{figures/small-icons_no_white_border.pdf}}}
\newcommand{\FVIII}{\raisebox{-4.5pt}{\includegraphics[page=7]{figures/small-icons_no_white_border.pdf}}}

\definecolor{customcyan}{RGB}{35,161,224}
\definecolor{custompurple}{RGB}{163,16,124}
\definecolor{customorange}{RGB}{223,155,27}

\newtheorem{question}{Question}

\newcommand{\naetsat}{\textsc{Mnae3Sat}\xspace}
\newcommand{\tcsc}{$t$-CSC\xspace}
\newcommand{\boldtcsc}{$\bm{t}$-CSC\xspace}
\newcommand{\titletcsc}{\texorpdfstring{\bm{$t$}}{t}-CSC\xspace}

\newcommand{\N}{\mathbb{N}}
\DeclarePairedDelimiterX\set[1]\lbrace\rbrace{\def\given{\;\delimsize\vert\;}#1}

\DeclareMathOperator{\circc}{cc}
\DeclareMathOperator{\ccross}{cross}

\newcommand{\omitted}{$\bigstar$}

\definecolor{IPEgold}{RGB}{255,215, 0}
\definecolor{purple}{RGB}{196,100,170}
\definecolor{cyan2}{RGB}{112,194,235}
\definecolor{orange2}{RGB}{235,190,107}
\definecolor{maygreen3}{RGB}{221,233, 197}
\definecolor{signalblue4}{RGB}{191,207, 240}
\definecolor{signalblue2}{RGB}{107,143, 240}
\definecolor{eurocgblue}{RGB}{150,190,240}
\definecolor{eurocgbluedark}{RGB}{90,115,145}

\makeatletter
\def\renewtheorem#1{%
  \expandafter\let\csname#1\endcsname\relax
  \expandafter\let\csname c@#1\endcsname\relax
  \gdef\renewtheorem@envname{#1}
  \renewtheorem@secpar
}
\def\renewtheorem@secpar{\@ifnextchar[{\renewtheorem@numberedlike}{\renewtheorem@nonumberedlike}}
\def\renewtheorem@numberedlike[#1]#2{\newtheorem{\renewtheorem@envname}[#1]{#2}}
\def\renewtheorem@nonumberedlike#1{
\def\renewtheorem@caption{#1}
\edef\renewtheorem@nowithin{\noexpand\newtheorem{\renewtheorem@envname}{\renewtheorem@caption}}
\renewtheorem@thirdpar
}
\def\renewtheorem@thirdpar{\@ifnextchar[{\renewtheorem@within}{\renewtheorem@nowithin}}
\def\renewtheorem@within[#1]{\renewtheorem@nowithin[#1]}
\makeatother

\crefname{observation}{Obs.\!}{Obs.\!}
\Crefname{observation}{Obs.\!}{Obs.\!}
\crefname{lemma}{Lem.}{Lem.\!}
\Crefname{lemma}{Lem.}{Lem.\!}
\crefname{proposition}{Prop.}{Prop.\!}
\Crefname{proposition}{Prop.}{Prop.\!}
\crefname{theorem}{Thm.}{Thm.\!}
\Crefname{theorem}{Thm.}{Thm.\!}
\crefname{corollary}{Cor.\!}{Cor.\!}
\Crefname{corollary}{Cor.\!}{Cor.\!}
\crefname{appendix}{App.\!}{App.\!}
\Crefname{appendix}{App.\!}{App.\!}
\crefname{section}{Sec.\!}{Sec.\!}
\Crefname{section}{Sec.\!}{Sec.\!}
\crefname{figure}{Fig.\!}{Fig.\!}
\Crefname{figure}{Fig.\!}{Fig.\!}
\crefname{table}{Table}{Table}
\Crefname{table}{Table}{Table}

\begin{document}

\maketitle

\begin{abstract}
In this work, we introduce \emph{$t$-colorable $k$-plane drawings}, that is, drawings of graphs with a $t$-edge-coloring where every edge is crossed by at most $k$~edges of each color.
We give tight upper bounds on the edge- and crossing density for small values of~$t$ and~$k$ and show that the recognition of such drawings is $\NP$-complete if $t\geq 2$ and $k\geq1$.
\subparagraph{Generative AI Declaration}
No generative AI was used at any point of this research project.
\end{abstract}

\begin{figure}[tbp]
    \savebox{\imagebox}{\includegraphics[page=7]{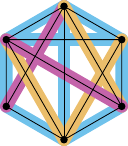}}
    \centering
    \subcaptionbox{\label{fig:ex_2-col_1-plane}}{\raisebox{\dimexpr.5\ht\imagebox-.5\height}{\includegraphics[page=5]{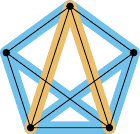}}}
    \hfil
    \subcaptionbox{\label{fig:ex_k_3_3}}{\raisebox{\dimexpr.5\ht\imagebox-.5\height}{\includegraphics[page=9]{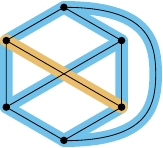}}}
    \hfil
    \subcaptionbox{\label{fig:ex_2-col_2-plane}}{\raisebox{\dimexpr.5\ht\imagebox-.5\height}{\includegraphics[page=8]{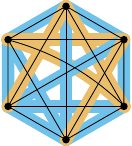}}}
    \hfil
    \subcaptionbox{\label{fig:ex_3-col_1-plane}}{\usebox{\imagebox}}
    \caption{Examples of $t$-colorable $k$-plane drawings for (a)-(b) $t=2, k=1$, (c) $t=k=2$, as well as (d)~$t=3, k=1$.}
    \label{fig:teaser}
\end{figure}

\section{Introduction}
 \label{sec:intro}

Beyond planar graphs have been studied extensively over the last decades.
A particular focus has been on \emph{$k$-plane drawings}, that is drawings of graphs where every edge is crossed at most $k$~times.
For $k \in \set{1,2,3}$ the structure of such drawings is well-understood \cite{bekos2023optimal-2-3-plane} and bounds on the edge density \cite{bodendiek1983bemerkungen, pach1996graphsFewCrossings, pach1997graphsFewCrossings, pach2004improving, pach2006improving} and crossing density \cite{bekos2024k-planar, goetze2025crossing} have been obtained.
While it is easy to check whether a given drawing is $k$-plane (for given $k$), it is $\NP$-hard to recognize graphs which admit such drawings \cite{urschel2021testingGapK-planarity}.

In this work, we consider similar questions for \emph{$t$-colorable $k$-plane drawings}, that is drawings that admit a $t$-edge-coloring such that each edge has at most $k$~crossings with edges of each color, see \cref{fig:teaser} for an illustration.
Such edge-colorings correspond to frugal colorings (see \cref{sec:prelim} for a definition) of
string graphs.
Upper bounds on the minimum number of colors needed in proper $k$-frugal colorings have been obtained previously for planar graphs \cite{aminiFrugalColouringGraphs2007} and graphs of bounded maximum degree \cite[Theorem~3.5]{kangFrugalAcyclicStar2011}. Recognizing graphs that admit proper $k$-frugal $t$-vertex colorings is $\NP$-hard for almost all $k,t \in \N$ \cite{bardComplexityFrugalColouring2021}.

We give upper bounds on the edge density for small~$t$ and~$k$ (see \cref{tab:overview_edge_density} for an overview) and on the crossing density of $2$-colorable $1$-plane drawings (cf. \cref{sec:crossing_density_2-col_1-plane}).
Further, we investigate the complexity of recognizing $t$-colorable $k$-plane drawings, see \cref{tab:overview_complexity} for an overview.
Results marked with \omitted{} are proven in the appendix.

\ra{1.3}
\def\smallDist{17}
\def\largeDist{40}
\begin{table}[tbp]
\centering
\subcaptionbox{\label{tab:overview_edge_density}}{
\centering
\begin{tabular*}{\linewidth}{c@{\extracolsep{0pt}} @{\extracolsep{\smallDist pt}}c@{\extracolsep{0pt}} @{\extracolsep{\smallDist pt}}c@{\extracolsep{0pt}} @{\extracolsep{\smallDist pt}}c@{\extracolsep{0pt}} @{\extracolsep{\smallDist pt}}}
 \toprule
        &  $1$-colorable & $2$-colorable & $3$-colorable  \\
        \midrule
        \arrayrulecolor{white}
        $k=1$  & \multicolumn{1}{|c|}{\cellcolor{lightgray}$4(n-2)$ \cite{bodendiek1983bemerkungen}} & $4.\overline{3}(n-2)$ \shortrefs{cor:edge_density_2-col_1-plane_upper}{prop:edge_density_2-col_1-plan_lower} & $5(n-2)$ \shortrefs{cor:edge_density_3-col_1-plane_upper}{prop:edge_density_3-col_1-plan_lower}
        \\
        \cmidrule{2-3}
        $k=2$  & \multicolumn{1}{|c|}{\cellcolor{lightgray}$5(n-2)$ \cite{pach1996graphsFewCrossings, pach1997graphsFewCrossings}} & \multicolumn{1}{|c|}{$6(n-2)$ \cite{ackerman2019topological,bungener2025simplified} \shortref{thm:edge-density_2-col_2-plane}}
        &   Open   \\
        \arrayrulecolor{black}
        \bottomrule
        \vspace{0.01cm}
\end{tabular*}}
\subcaptionbox{\label{tab:overview_complexity}}{\begin{tabular*}{\linewidth}{c@{\extracolsep{0pt}} @{\extracolsep{\smallDist pt}}c@{\extracolsep{0pt}} @{\extracolsep{\smallDist pt}}c@{\extracolsep{0pt}} @{\extracolsep{\smallDist pt}}c@{\extracolsep{0pt}} @{\extracolsep{\smallDist pt}}c@{\extracolsep{0pt}} @{\extracolsep{\smallDist pt}}}
        \toprule
        & $1$-colorable & $2$-colorable & $3$-colorable & $(t>3)$-colorable \\
        \midrule
        \arrayrulecolor{white}
        $k=1$  & trivial & $\mathcal{O}(n+m)$ \shortref{thm:recognition_2_1} & $\NPC$ \shortref{thm:recognition_3_1} & $\NPC$ \shortref{thm:recognition_k_1} \\
        \cmidrule{2-4}
        $k=2$  & trivial & $\NPC$ \shortref{thm:recognition_2_2} & {\cellcolor{eurocgblue}} $\NPC$ \shortref{thm:recognition_collected} & {\cellcolor{eurocgblue}} $\NPC$ \shortref{thm:recognition_collected} \\
        \cmidrule{2-4}
        $k > 2$  & trivial & {\cellcolor{eurocgblue}} $\NPC$ \shortref{thm:recognition_collected} & {\cellcolor{eurocgblue}} $\NPC$ \shortref{thm:recognition_collected} &  {\cellcolor{eurocgblue}} $\NPC$ \shortref{thm:recognition_collected} \\
        \arrayrulecolor{black}
        \bottomrule
    \end{tabular*}}
\smallskip
\caption{
    Overview of results for $t$-colorable $k$-plane drawings.\;\;\textcolor{lightgray}{$\blacksquare$} previously known\;\;\textcolor{eurocgblue}{$\blacksquare$} implied by \cref{thm:recognition_collected}.
    (a) Maximum number of edges for drawings on $n$ vertices.
  (b) Recognition complexity.}
\end{table}

\subparagraph*{Related Work.}
The thickness of a graph~$G$ is the smallest number~$t$ of colors such that the edges of~$G$ can be decomposed into $t$~planar graphs.
In fact, if~$G$ has thickness~$t$, it admits a drawing~$\Gamma$ and a $t$-edge-coloring such that each color class of~$\Gamma$ is plane
\cite[p.\,95]{duncanGraphThicknessGeometric2011}.
We refer to \cite{mutzelThicknessGraphsSurvey1998} for an overview of different variants of thickness, each of which restricts the drawing~$\Gamma$ or the planar graphs with which we cover the edges.
A variant where each color class is $k$-plane does not seem to have been studied.
However, we emphasize that for a drawing to be $t$-colored $k$-plane, it is not enough to require that each color class is $k$-plane.
Consider for example a $2$-edge-coloring of a  drawing~$\Gamma$ of~$K_5$ where each color class is plane, see \cref{fig:ex_k_3_3}.
While the drawing~$\Gamma$ contains no monochromatic crossings, it does contain multi-colored crossings.
As such, $\Gamma$ is $2$-colorable $1$-plane, but not $t$-colorable $0$-plane for any $t$.

\section{Preliminaries and notation}\label{sec:prelim}

We consider drawings in the plane, where vertices are represented by points and edges by simple Jordan curves.
Edges only cross in internal points, do not share an infinite number of points or touch (sharing a point without crossing), edges do not self-intersect and no more than two edges cross in a point, neither may adjacent edges cross.

The drawings we consider may contain parallel edges (but no loops).
In general, the number of edges may be arbitrarily large.
However, we only consider \emph{non-homotopic} drawings.
A drawing is non-homotopic if every region whose boundary forms a \emph{lens}, i.e.\ it is formed by (parts) of exactly two edges, contains a vertex or a crossing in its interior, see \cref{fig:lenses} for examples.
An edge is \emph{plane} if it is uncrossed in~$\Gamma$.
We call the subdrawing of a drawing~$\Gamma$ consisting of all plane edges the \emph{plane skeleton} of~$\Gamma$.
\begin{figure}
    \centering
    \subcaptionbox{\label{fig:example_lens-1}}{\includegraphics[page=1]{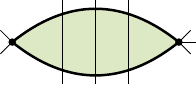}}
    \hfil
    \subcaptionbox{\label{fig:example_lens-2}}{\includegraphics[page=2]{figures/examples_lenses.pdf}}
    \hfil
    \subcaptionbox{\label{fig:example_lens-3}}{\includegraphics[page=3]{figures/examples_lenses.pdf}}
    \caption{Examples of lenses (lightgreen) with their boundary (thick). An empty lens (a) cannot appear in a non-homotopic drawing, while lenses that contain a crossing or a vertex  (b)-(c) may. }
    \label{fig:lenses}
\end{figure}

\subparagraph*{\texorpdfstring{$\bm{t}$}{t}-colorable \texorpdfstring{$\bm{k}$}{k}-plane drawings.}
We call a graph~$G$ \emph{$t$-colorable $k$-planar} if it admits a drawing~$\Gamma$ and a (not necessarily proper) $t$-edge-coloring~$\varphi\colon E(G) \to [t]$ such that each edge is crossed by at most $k$~edges in~$\Gamma$ for each of the $t$~colors.
(As usual $[t] = \{1, \ldots, t\}$.)
The drawing~$\Gamma$ is called \emph{$t$-colorable $k$-plane} and its witness a \emph{$k$-plane} $t$-edge-coloring.
\begin{observation}\label{obs:kt-plane}
Every $t$-colorable $k$-plane drawing is $kt$-plane.
\end{observation}
Yet, we will see that there are $kt$-plane drawings that are not $t$-colorable $k$-plane (cf. \cref{thm:tk-plane_not_t-col_k-plane}).

\subparagraph*{Segments, cells and configurations.}
An edge with  $\ell$ crossings is split into $\ell +1$ parts.
The parts incident to a vertex are called \emph{outer segments}, the other \emph{inner} segments.
The labeled embedded graph obtained by replacing each crossing with a vertex is the \emph{planarization $\Lambda(\Gamma)$} of drawing~$\Gamma$.
The labels indicate whether a vertex of~$\Lambda(\Gamma)$ corresponds to a~vertex~or a crossing of~$\Gamma$.
Regions of~$\Gamma$ that correspond to faces of~$\Lambda(\Gamma)$ are \emph{cells}, denoted by~$C(\Gamma)$.
The boundary~$\partial c$ of a cell~$c \in C(\Gamma)$ consists of a set of closed walks.
A walk may contain the same segment twice.
A \emph{configuration} is a labeled embedded subgraph of the planarization~$\Lambda(\Gamma)$.
Intuitively, it is an arrangement of one or several adjacent cells of~$\Gamma$.
Two configurations are of the same \emph{type} if they are isomorphic as labeled embedded subgraphs (up to reflection).
We denote different types by small pictograms such as $\FVItIII$ or~$\FVtII$, see \cref{fig:example_configuration} for an illustration.\noindent
\begin{figure}
    \centering
    \includegraphics[page=2]{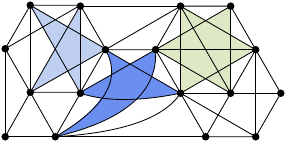}
    \caption{A drawing that contains a $\FVtII$- (\textcolor{signalblue2}{$\blacksquare$}) and a $\FVItIII$-configuration (\textcolor{maygreen3}{$\blacksquare$}). While another region~(\textcolor{signalblue4}{$\blacksquare$}) contains a subdrawing isomorphic to a $\FVtII$-graph, it is no $\FVtII$-configuration. }
    \label{fig:example_configuration}
\end{figure}

\subparagraph*{\texorpdfstring{$\bm{k}$}{k}-frugal colorings.}
\label{par:frugal_colorings}
For a drawing~$\Gamma$, consider the \emph{conflict graph}~$H$ with vertices $V(H) = E(\Gamma)$ corresponding to the edges of~$\Gamma$.
Two vertices in~$V(H)$ are joined by an edge in $E(H)$ if the corresponding edges~$e,e'$ cross in~$\Gamma$.
If~$\Gamma$ is $t$-colorable $k$-plane, its edge-coloring corresponds to a vertex-coloring~$\Phi\colon V(H) \to [t]$ where each color class appears at most $k$~times in the neighborhood of each vertex,
that is~$\abs{N(v) \cap \Phi^{-1}(i)} \leq k$ for every~$v \in V(H)$ and each color~$i \in [t]$.
Note that~$\Phi$ need not be proper.
Such a coloring of a graph~$H$ is called \emph{$k$-frugal}.
The minimum number of colors~$\varphi_k(H)$ among all $k$-frugal colorings of~$H$ is the \emph{$k$-frugal chromatic number}.

A $k$-plane $t$-edge-coloring of a drawing~$\Gamma$ corresponds to a $k$-frugal $t$-vertex-coloring of its conflict graph~$H$.
Hence, we are interested in (not necessarily proper) $k$-frugal colorings of \emph{string graphs}, i.e., intersection graphs of Jordan curves.
\begin{observation}
\label{obs:correspondence_frugal_conflict_graph}
    A drawing~$\Gamma$ is $t$-colorable $k$-plane if and only if its conflict graph~$H$ admits a $k$-frugal $t$-vertex-coloring.
\end{observation}
\begin{observation}
\label{obs:frugal_lower_bound}
    For a graph~$H$ of maximum degree~$\Delta(H)$, we have $\varphi_k(H) \geq \frac{\Delta(H)}{k}$.
\end{observation}

Proper frugal colorings have been introduced by Hind, Molloy and Reed \cite{hindColouringGraphFrugally1997}.
They relate the proper $k$-frugal chromatic number to the maximum degree~$\Delta(H)$ of~$H$: every graph~$H$ admits a proper $(\Delta(H)+1)$-vertex-coloring. Yet, what is the minimum~$k$ such that some proper $(\Delta(H)+1)$-vertex-coloring is $k$-frugal?
The authors give upper bounds on~$k$ in terms of~$\Delta(H)$.
An asymptotic upper bound provided by Molloy and Reed \cite{molloyAsymptoticallyOptimalFrugal2010} is known to be tight \cite{hindColouringGraphFrugally1997}.
Amini, Esperet and van den Heuvel provide upper bounds on the proper $k$-frugal chromatic number for planar graphs \cite{aminiFrugalColouringGraphs2007}.
For $k=1$, many more results are known as proper $1$-frugal vertex-colorings coincide with proper vertex colorings of the square of~$H$, i.e.\ the graph on the same vertex set with edges between vertices at distance at most~$2$ in~$H$. See \cite[p.\,26]{chenGraphRhuedColorings2022} for an overview on colorings of the square.
For \emph{proper} $k$-frugal colorings, the computational complexity has been studied:
Bard, MacGillivray and Redlin characterize all fixed pairs of integers~$t,k \in \N$ such that deciding whether a given graph~$H$ admits a proper $k$-frugal $t$-vertex-coloring is~$\NP$-hard \cite[Theorem~1.2]{bardComplexityFrugalColouring2021}.
Dropping the condition of the coloring to be proper, Kang and Müller expand the study of frugal colorings \cite{kangFrugalAcyclicStar2011}.

\section{Edge density of \texorpdfstring{$\bm{t}$}{t}-colorable \texorpdfstring{$\bm{k}$}{k}-planar graphs for small \texorpdfstring{$\bm{t}$}{t} and \texorpdfstring{$\bm{k}$}{k}}
\label{sec:edge_density}
In this section, we take a closer look at the maximum number of edges in $t$-colorable $k$-planar graphs for $kt \leq 4$.
Recall that every such graph is in particular $kt$-planar (cf. \cref{obs:kt-plane}).
For $kt \leq 3$, some configurations appear in every sufficiently dense $kt$-plane drawing \cite{buengener2024improvingCrossingNumberDense2Planar3Planar}.
Yet, they may not be part of any $t$-colorable $k$-plane drawing due to the coloring restriction.
For $t=2, k=1$, this holds for $\FVtII$- and $\FVItII$-configurations (cf. \cref{sec:edge_density_2-col_1-plane}), and for $t=3,k=1$ for all $\FVItIII$-configurations (\cref{sec:edge_density_3-col_1-plane}).
The densest $4$-plane drawings turn out to be $2$-colorable $2$-plane~(\cref{sec:edge_density_2-col_2-plane}).
In each setting, this yields an upper bound on the edge density, with a matching lower bound based on modifications of the densest $kt$-plane drawings.
\begin{theorem}
\label{thm:edge_density_overview}
    Let~$\Gamma$ be a non-homotopic drawing of a graph on $n$ vertices.
    \begin{enumerate}[(i)]
        \item\label{itm:edge_density_2-col_1-plane} If~$\Gamma$ is $2$-colorable $1$-plane, then $\abs{E(\Gamma)} \leq 4.\overline{3}(n-2)$.
        \item If~$\Gamma$ is $3$-colorable $1$-plane, then $\abs{E(\Gamma)} \leq 5(n-2)$.
        \item If~$\Gamma$ is $2$-colorable $2$-plane, then $\abs{E(\Gamma)} \leq 6(n-2)$.
    \end{enumerate}
    In each case, there exists an infinite family of drawings which matches the upper bound.
\end{theorem}

\subsection{\texorpdfstring{$\bm{2}$}{2}-colorable \texorpdfstring{$\bm{1}$}{1}-plane drawings}
\label{sec:edge_density_2-col_1-plane}

For a cell $c \in C(\Gamma)$ of a $2$-colorable $1$-plane drawing~$\Gamma$, we call a closed walk~$w$ of $\partial c$ an \emph{inner cycle} if the component of the conflict graph containing the crossings of~$w$ is a cycle,~see \cref{fig:2-col_1-plane_pattern}.
 \begin{figure}
        \centering
        \includegraphics[page=2]{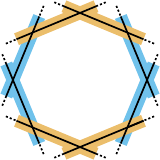}
        \caption{The coloring of an inner cycle (of size~$8$) in a $2$-colorable $1$-plane drawing.}
        \label{fig:2-col_1-plane_pattern}
\end{figure}
The size of an inner cycle $w$ is the number of incident segments.
In a $2$-colorable $1$-plane drawing,~each edge~$e$ that belongs to an inner cycle~$w$ is crossed by two edges~$e'$ and~$e''$ of different colors.

\begin{observation}
\label{obs:2-plane_inner_cell}
    Let~$\Gamma$ be a connected $2$-colorable $1$-plane drawing (in blue and orange) with an inner cycle~$w$ of size~$s$.
    The segments of~$w$ (in the order on the boundary) are colored with the pattern blue, blue, orange, orange; see \cref{fig:2-col_1-plane_pattern}.
    In particular, $s \equiv 0 \mod 4$.
\end{observation}

Büngener and Kaufmann relate the edge density of $2$-plane drawings to the number of drawings of subgraphs of type~$\FVtII$ and~$\FVItII$ \cite{buengener2024improvingCrossingNumberDense2Planar3Planar}.
Such subdrawings cannot be crossed in $2$-plane drawings and thus correspond to configurations of the same type.
While inner cycles of size~$5$ and~$6$ do not appear in $2$-colorable $1$-plane drawings, Büngener and Kaufmann show that such cycles cannot be avoided in dense $2$-plane drawings:

\begin{theorem}[{Büngener and Kaufmann \cite[Corollary~2]{buengener2024improvingCrossingNumberDense2Planar3Planar}}]
\label{thm:oversaturation_2_planar}
    Every non-homotopic $2$-plane drawing of a graph on $n$~vertices with $4.\overline{3}(n-2)+x$ edges for $x \in [0, 0.\overline{6}(n-2)]$ contains at least $x$ configurations of type~$\FVtII$ or~$\FVItII$.
\end{theorem}

Indeed, $\FVtII$- and $\FVItII$-configurations contain an inner cycle of size~$5$ and~$6$ respectively.
That is, such configurations do not appear in any $2$-colorable $1$-plane drawing.

\begin{corollary}
\label{cor:edge_density_2-col_1-plane_upper}
    Every non-homotopic $2$-colorable $1$-plane drawing~$\Gamma$ of a graph on $n$ vertices contains at most $4.\overline{3}(n-2)$ edges.
\end{corollary}

A small modification of the known densest $2$-plane drawings~\cite{pach1997graphsFewCrossings} yields a matching lower bound on the edge density of $2$-colorable $1$-plane drawings (see \cref{fig:edge_density_2-col_1-plan_lower}).

\begin{figure}
    \savebox{\imagebox}{\includegraphics[page=2]{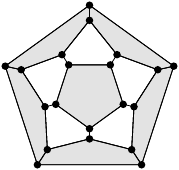}}
    \centering
    \subcaptionbox{\label{fig:edge_density_2-col_1-plan_lower-1}}{\usebox{\imagebox}}
    \hfil
    \subcaptionbox{\label{fig:edge_density_2-col_1-plan_lower-2}}{\raisebox{\dimexpr.5\ht\imagebox-.5\height}{\includegraphics[page=3]{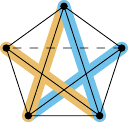}}}
     \caption{A $2$-colorable $1$-plane drawing based on~\cite[Figure~3]{pach1997graphsFewCrossings}.
    (a) A plane drawing with pentagonal faces. (b) To each pentagonal face all but one diagonal (dashed) are added.}
    \label{fig:edge_density_2-col_1-plan_lower}
\end{figure}
\begin{proposition}
\label{prop:edge_density_2-col_1-plan_lower}
    For every~$n = 20+15\ell$ with $\ell \in \mathbb{N}^+$, there exists a $2$-colorable $1$-plane drawing of a simple graph on~$n$ vertices with $4.\overline{3}(n-2)$~edges and $1.\overline{3}(n-2)$~crossings.
\end{proposition}
\begin{proof}
    Consider a plane drawing~$\Gamma_0$ on $n$~vertices where each face is bounded by a $5$-cycle.
    Adding all but one diagonal to each of the faces yields a $2$-colorable $1$-plane drawing~$\Gamma$, see \cref{fig:edge_density_2-col_1-plan_lower}.
    As $\Gamma_0$ is connected, we obtain with Euler's formula that~$\Gamma_0$ contains $f = \frac{2}{3}(n-2)$~faces and $m = \frac{5}{3}(n-2)$ edges.
    The bound on the number of edges and crossings of~$\Gamma$ now follows as each face of~$\Gamma_0$ contains two crossings and four edges of~$E(\Gamma) - E(\Gamma_0)$.

    Such drawings~$\Gamma_0$ (and thus also~$\Gamma$) exist for every $n = 20 + 15\ell$ with $\ell \in \N_0$.
    For $\ell=0$, consider the drawing in \cref{fig:edge_density_2-col_1-plan_lower-1}.
    For larger values of~$\ell$, identify the outer face of the drawing for~$\ell=0$ with a face of a drawing for~$\ell-1$.
\end{proof}

\subsection{\texorpdfstring{$\bm{3}$}{3}-colorable \texorpdfstring{$\bm{1}$}{1}-plane drawings}
\label{sec:edge_density_3-col_1-plane}
While $2$-colorable $1$-plane drawings contain no inner cycles of size~$5$ and~$6$, $3$-colorable $1$-plane drawings avoid $\FVItIII$-configurations:
\begin{restatable}{lemma}{NoFSix}
\label{lem:no_f6t3_in_3-col_1_plane}
A $3$-colorable $1$-plane drawing contains no $\FVItIII$-configuration.
\end{restatable}
\begin{proof}
Suppose there exists a $1$-plane $3$-edge-coloring~$\varphi\colon E(C) \to \set{1,2,3}$ of the $\FVItIII$-{}con\-fi\-gu\-ra\-tion $C$.
Let~$v_0, \dots, v_5$ be the ordering of the vertices along the outer face of~$C$.
We distinguish two types of edges in~$C$:
\begin{itemize}
    \item a \emph{short diagonal} connects two vertices~$v_i,v_j$ with $i-j \equiv 2 \mod 6$ or $j-i \equiv 2 \mod 6$,
    \item a \emph{long diagonal} connects two vertices~$v_i,v_j$ with $i-j \equiv 3 \mod 6$.
\end{itemize}
The short diagonals form two independent triangles~$T_1,T_2$ that intersect such that each edge is crossed twice.
\begin{figure}[tbp]
    \centering
    \subcaptionbox{\label{fig:f6_3_not_3-col_1-plane-1}}{\includegraphics[page=1]{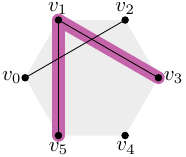}}
    \hfil
    \subcaptionbox{\label{fig:f6_3_not_3-col_1-plane-2}}{\includegraphics[page=2]{figures/f6_3_not_3-col_1-plane.pdf}}
    \hfil
    \subcaptionbox{\label{fig:f6_3_not_3-col_1-plane-3}}{\includegraphics[page=3]{figures/f6_3_not_3-col_1-plane.pdf}}
    \caption{(a)-(b) If two adjacent edges~$e,e'$ (\textcolor{purple}{$\blacksquare$}) in~$\FVItIII$ have the same color, the $3$-edge-coloring is not $1$-plane as a short diagonal is crossed by~$e$ and~$e'$. (c) The only $1$-plane $3$-edge-coloring of~$T_1 \cup \set{d_1,d_2}$.}
    \label{fig:f6_3_not_3-col_1-plane}
\end{figure}
We first observe that no two adjacent edges~$e=v_iv_j, e'=v_jv_k$ in~$C$ have the same color.
Indeed, otherwise the short diagonal~$v_{j-1}v_{j+2}$ is crossed by~$e$ and~$e'$ certifying that the $3$-edge-coloring~$\varphi$ is not $1$-plane, see \cref{fig:f6_3_not_3-col_1-plane-1} and \ref{fig:f6_3_not_3-col_1-plane-2}.

That is, at each vertex at most one edge of each color is incident.
Restricting~$C$ to the two long diagonals~$d_1,d_2$ and the short diagonals of the triangle~$T_1$, we are in the situation depicted in \cref{fig:f6_3_not_3-col_1-plane-3}.
Each of the two long diagonals is crossed by a blue and an orange edge.
Thus, two of the remaining short diagonals (all of which belong to the triangle~$T_2$) are red.
In particular, we have two adjacent red edges, contradicting the observation above.
\end{proof}
Yet, every dense $3$-plane drawing contains a $\FVItIII$-configuration:
\begin{theorem}[{Büngener and Kaufmann \cite[Corollary~4]{buengener2024improvingCrossingNumberDense2Planar3Planar}}]
\label{thm:oversaturation_3_planar}
    Every non-homotopic $3$-plane drawing of a graph on $n$~vertices with $5(n-2)+x$ edges for $x \in [0, 0.5(n-2)]$ contains at least $x$ configurations of type~$\FVItIII$.
\end{theorem}
As no non-homotopic $3$-colorable $1$-plane drawing contains a $\FVItIII$-configuration by \cref{lem:no_f6t3_in_3-col_1_plane},
\cref{thm:oversaturation_3_planar} yields an upper bound on the number of edges.
\begin{corollary}
\label{cor:edge_density_3-col_1-plane_upper}
    Every non-homotopic $3$-colorable $1$-plane drawing~$\Gamma$ of a graph on~$n$ vertices contains at most~$5(n-2)$ edges.
\end{corollary}
In fact, this upper bound is tight.
To construct such a drawing, consider a plane drawing where every face is bounded by a $6$-cycle.
Such plane drawings with~$f$ faces can be obtained by subdividing twice each edge of a multi-graph on two vertices and $f$~parallel edges.
Within each face, we can add seven edges such that the resulting drawing~$\Gamma$ is $3$-colorable $1$-plane, see \cref{fig:edge_density_3-col_1-plan_lower} for an illustration.
That is, each face of the plane skeleton of~$\Gamma$ (the subdrawing on all plane edges) now contains seven edges and eight crossings.
Double-counting the edge-face-incidences of the plane skeleton yields the following:
\begin{figure}
    \centering
    \subcaptionbox{\label{fig:cylinder}}{\includegraphics[page=3]{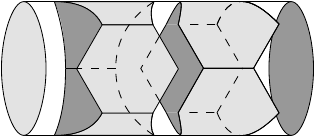}}
    \hfil
    \subcaptionbox{\label{fig:edge_density_3-col_1-plan_lower}}{\includegraphics[page=2]{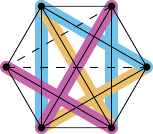}}
    \hfil
    \subcaptionbox{\label{fig:edge_density_2-col_2-plan_lower}}{\includegraphics[page=2]{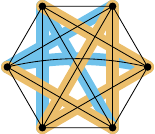}}
  \caption{Construction of $t$-colorable $k$-plane drawings with maximum edge density (figure based on \cite[Figure~35]{ackerman2019topological}). (a) A cylinder with two layers, each consisting of three hexagonal faces. (b) For $t=3,k=1$, to each face of the cylinder all diagonals but two are added. Missing diagonals are dashed. (c) For $t=2,k=2$, to each face of \cref{fig:cylinder} all diagonals are added.}
    \label{fig:edge_density_lower}
\end{figure}
\begin{proposition}
\label{prop:edge_density_3-col_1-plan_lower}
    For every $n = 6+2\ell$ with~$\ell \in \N_0$, there exists a non-homotopic $3$-colorable $1$-plane drawing of a multi-graph on $n$ vertices with $5(n-2)$~edges and $4(n-2)$~crossings.
\end{proposition}
In fact, the drawings constructed above are obtained by deleting (specific) edges from $3$-plane non-homotopic drawings of maximum edge density:
Such a drawing~$\Gamma$ on $n$~vertices contains $5.5(n-2)$~edges \cite{pach2004improving,pach2006improving,goetze2025crossing}\footnote{The upper bound in \cite{pach2004improving} and \cite{pach2006improving} is only given for simple graphs.} and admits a non-homotopic $3$-plane drawing where every face of the plane skeleton~$\Gamma_0$ is a $6$-cycle, each containing a $\FVItIII$-configuration \cite[Theorem~17]{bekos2023optimal-2-3-plane}.

\subsection{\texorpdfstring{$\bm{2}$}{2}-colorable \texorpdfstring{$\bm{2}$}{2}-plane drawings}
\label{sec:edge_density_2-col_2-plane}
Recall that every $2$-colorable $2$-plane drawing is in particular $4$-plane (cf. \cref{obs:kt-plane}).
That is, it contains at most $6n-12$ edges \cite[Theorem~4]{ackerman2019topological} \cite{bungener2025simplified}\footnote{Ackerman provides a proof for simple graphs in \cite{ackerman2019topological} which was recently generalized by Büngener to non-homotopic drawings.}.
In fact, the known $4$-plane drawings with maximum edge density \cite[p.\,37]{ackerman2019topological} are also $2$-colorable $2$-plane, see \cref{fig:edge_density_2-col_2-plan_lower}.
Such a drawing is obtained from a plane drawing~$\Gamma_0$ where every face is a $6$-cycle by inserting all nine diagonals, i.e. each face of~$\Gamma_0$ contains nine edges and fifteen crossings.

\begin{theorem}
\label{thm:edge-density_2-col_2-plane}
    Every non-homotopic $2$-colorable $2$-plane drawing of a multigraph graph on~$n \geq 3$ vertices contains at most $6(n-2)$~edges and for every $n = 6 + 2\ell$ with~$\ell \in \N_0$, there exists such a drawing with~$6(n-2)$ edges.
\end{theorem}

Yet, while every $2$-colorable $2$-plane drawing is $4$-plane, there are $4$-plane drawings that are not $2$-colorable $2$-plane.
In fact, there are also $3$-plane drawings that are not $2$-colorable $2$-plane.
The $3$-regular graph~$H$ represented in \cref{fig:deg_3_no_2_frugal_2-coloring} is the conflict graph of a $3$-plane drawing as $H$ is planar (i.e., in particular a string graph).
As $H$ admits no $2$-frugal $2$-coloring, the corresponding drawing is not $2$-colorable $2$-plane (cf. \cref{obs:frugal_lower_bound}).
\begin{figure}[tbp]
    \centering
    \includegraphics[page=1]{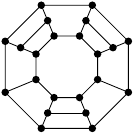}
    \caption{A planar $3$-regular graph with no $2$-frugal $2$-coloring, found by automatically checking examples provided by \href{https://houseofgraphs.org/graphs/53789}{House of Graphs}. The graph is taken from
    \cite[p.\,88]{bickle2023plane}.}
    \label{fig:deg_3_no_2_frugal_2-coloring}
\end{figure}

\section{Crossing density of \texorpdfstring{$\bm{2}$}{2}-colorable \texorpdfstring{$\bm{1}$}{1}-plane drawings}
\label{sec:crossing_density_2-col_1-plane}

The aim of this section is to prove the following:
\begin{restatable}{theorem}{twoColOnePlaneCrossings}
\label{thm:crossings_in_2-col_1-plane}
    Every filled non-homotopic $2$-colorable $1$-plane drawing~$\Gamma$ of a graph on $n$~vertices with $n \geq 5$ has at most $2.\overline{6}(n-2)$~crossings.
\end{restatable}
In fact, it is enough to prove \cref{thm:crossings_in_2-col_1-plane} for what we call \emph{filled} drawings, a concept used in previous work, e.g.~\cite{goetze2025crossing}.
A drawing~$\Gamma$ is \emph{filled} if for every cell~$c$ of~$\Gamma$, every pair of (distinct) vertices~$u,v$ on the boundary~$\partial c$ is connected with a plane edge~$uv$ on~$\partial c$.

\begin{observation}
    Every non-homotopic $2$-colorable $1$-plane drawing~$\Gamma$ is contained in a filled non-homotopic $2$-colorable $1$-plane drawing~$\Gamma'$.
\end{observation}
Yet, even if~$\Gamma$ is a drawing of a simple graph, $\Gamma'$ may contain parallel edges.
In order to prove \cref{thm:crossings_in_2-col_1-plane}, we first provide an upper bound on the number of crossings within faces of the plane skeleton~$\Gamma_0$.
This bound is based on the function $\ccross\colon \N \to \N_0$ with
\begin{align*}
        \ccross(x) =
        \begin{cases}
            0, &\text{if $x \leq 3$,} \\
            2, &\text{if $x = 4$,} \\
            x-1, &\text{if $x \geq 5$ and~$x \nequiv 0 \mod 4$}, \\
            x, &\text{if $x \geq 5$ and~$x \equiv 0 \mod 4$},
        \end{cases}
\end{align*}
for all~$x$. We call $x$ the \emph{circumference} of faces.
Each edge of~$\Gamma_0$ has two \emph{sides}, both of which may be incident to the same face~$f \in C(\Gamma_0)$.
The \emph{circumference}~$\circc(f)$ of a face~$f \in C(\Gamma_0)$ is the number of sides incident to~$f$, see \cref{fig:example_circumference} for examples.
\begin{figure}
    \savebox{\imagebox}{\includegraphics[page=3]{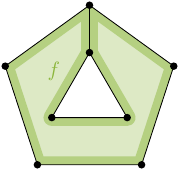}}
    \centering
    \subcaptionbox{\label{fig:example_circumference-1}}{\usebox{\imagebox}}
    \hfil
    \subcaptionbox{\label{fig:example_circumference-2}}{\raisebox{\dimexpr.5\ht\imagebox-.5\height}{\includegraphics[page=4]{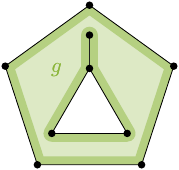}}}
    \caption{
    Examples where faces $f$ (a) and $g$ (b) have both circumference $\circc(f)=\circc(g)=10$.}
    \label{fig:example_circumference}
\end{figure}

\begin{restatable}{lemma}{TwoOnePlaneCrossingInFace}
\label{lem:2-col_1-plane_crossings_in_face}
    In a filled non-homotopic $2$-colorable $1$-plane drawing~$\Gamma$ on $n \geq 5$~vertices, every face~$f$ of the plane skeleton~$\Gamma_0$ contains at most $\ccross(\circc(f))$ crossings.
\end{restatable}
\begin{proof}
    Let~$X(f)$ denote the set of crossings inside the face~$f$.
    As the drawing is $2$-plane, each crossing~$x \in X(f)$ is incident to two outer segments with distinct endpoints~$u$ and~$v$.
    The endpoints~$u$ and~$v$ lie on a common cell~$c$ of~$\Gamma$.
    That is, they are joined by an edge~$uv$ on the boundary of~$f$ as the drawing is filled.
    We say that the side of the edge~$uv$ incident to~$c$ is \emph{covered} by the crossing~$x$.
    Note in particular that no other crossing covers the same side of~$uv$.
    As each crossing within~$f$ covers a (distinct) side of an edge of~$f$, we obtain
    $\abs{X(f)} \leq \circc(f)$.

    If~$\circc(f) \leq 3$, then~$f$ is incident to at most three vertices.
    As crossing edges do not share an endpoint, we obtain~$\abs{X(f)} = 0 = \ccross(\circc(f))$.

    If~$\circc(f) = 4$, then~$f$ is either incident to
    \begin{enumerate}
    \item two edges, both of which have all sides on~$f$,
    \item three edges, where two have only one side on~$f$ and one edge has both sides on~$f$,
    \item or four edges, all of which only have one side on~$f$.
    \end{enumerate}
    In the first case, $\Gamma_0$ has no other face than~$f$, see \cref{fig:small_face_crossings-1} for an illustration.
    In particular, $\Gamma$ contains at most four vertices and \cref{lem:2-col_1-plane_crossings_in_face} does not apply since we assume that~$\abs{V(\Gamma)} \geq 5$.
    In the second case, $f$ contains at most two crossings, see \cref{fig:small_face_crossings-2}-\cref{fig:small_face_crossings-3} and we have
    $\abs{X(f)} \leq 2 \leq \ccross(\circc(f))$.
    In the third case, the boundary of~$f$ is a $4$-cycle.
    That is,~$f$ contains at most two crossing edges and~$\abs{X(f)} \leq 1 \leq \ccross(\circc(f))$, see \cref{fig:small_face_crossings-4}.

    Otherwise~$\circc(f) \geq 5$.
    Recall that $\abs{X(f)} \leq \circc(f)$.
    Now suppose $\abs{X(f)} = \circc(f)$.
    It remains to show that $\circc(f) \equiv 0 \pmod 4$.

    We first argue that $\abs{X(f)} = \abs{S_{\mathrm{inner}}(f)}$ where~$S_{\mathrm{inner}}(f)$ denotes the set of inner segments within~$f$.
    Note that all sides of~$f$ are covered.
    That is, exactly one of the four cells of~$\Gamma$ incident to a crossing~$x \in X(f)$ is incident to a side of~$f$.
    In particular, each crossing~$x \in X(f)$ is incident to two inner segments, and each inner segment~$s \in S_{\mathrm{inner}}(f)$ to two crossings.
    Double-counting the incidences between~$X(f)$ and~$S_{\mathrm{inner}}(f)$ yields $2\abs{X(f)} = 2\abs{S_{\mathrm{inner}}(f)}$.

    Consider the auxiliary graph~$H$ whose vertices correspond to the inner segments~$S_{\mathrm{inner}}(f)$ and where two vertices are connected if the corresponding segments are incident to the same crossing.
    Each vertex in~$H$ has degree~$2$.
    Thus,~$H$ is a disjoint union of cycles, each corresponding to an inner cycle of~$\Gamma$ within~$f$.
    By \cref{obs:2-plane_inner_cell}, each inner cycle of size~$s$ satisfies~$s \equiv 0 \mod 4$.
    In particular, we obtain $\circc(f) \equiv \abs{X(f)} \equiv \abs{S_{\mathrm{inner}}(f)} \equiv 0 \mod 4$ which concludes the proof.
\end{proof}

\begin{figure}
    \centering
    \savebox{\imagebox}{\includegraphics[page=1]{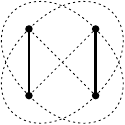}}
    \subcaptionbox{\label{fig:small_face_crossings-1}}{\usebox{\imagebox}}
    \hfil
    \subcaptionbox{\label{fig:small_face_crossings-2}}{\includegraphics[page=2]{figures/small_face_crossings.pdf}}
    \hfil
    \subcaptionbox{\label{fig:small_face_crossings-3}}{\includegraphics[page=3]{figures/small_face_crossings.pdf}}
    \hfil
    \subcaptionbox{\label{fig:small_face_crossings-4}}{\raisebox{\dimexpr.5\ht\imagebox-.5\height}{\includegraphics[page=4]{figures/small_face_crossings.pdf}}}
    \caption{A face~$f$ of the plane skeleton (thick) of a $2$-colorable $1$-plane drawing with~$\circc(f) = 4$. Edges inside~$f$ are represented with dotted lines. (a) If~$f$ is only incident to two edges, $f$ is the only face but might contain many crossings. (b)-(c) If~$f$ is incident to three edges, then $f$ is a $2$-cycle with an independent edge or a pendant edge inside. (d) Otherwise, $f$ is a $4$-cycle and contains at most one crossing.}
\end{figure}

Yet, the same bound does not apply to non-filled drawings, see \cref{fig:example_face_with_too_many_crossings}.
\begin{figure}
    \centering
    \savebox{\imagebox}{\includegraphics[page=2]{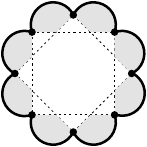}}
    \subcaptionbox{\label{fig:example_face_with_too_many_crossings-1}}{\usebox{\imagebox}}
    \hfil
    \subcaptionbox{\label{fig:example_face_with_too_many_crossings-2}}{\includegraphics[page=3]{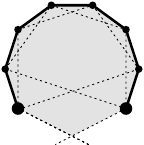}}
    \caption{A (non-filled) $2$-colorable $1$-plane drawing~$\Gamma$ with plane edges (thick) and crossed edges (dotted). Each gray area in (a) contains a crossing component consisting of eight edges as depicted in (b). The inner face~$f \in C(\Gamma_0)$ in (a)
    with $\circc(f) = 8+8\cdot6 = 56$ contains $8+8\cdot 8= 72$~crossings.}
    \label{fig:example_face_with_too_many_crossings}
\end{figure}
For filled drawings, we now obtain an upper bound on the number of crossings.
\twoColOnePlaneCrossings*
\begin{proof}
    Let~$a_i$ be the number of faces~$f$ in the plane skeleton~$\Gamma_0$ with~$\circc(f) = i$ and let~$X$ be the set of crossings in~$\Gamma$.
    For~$i \geq 3$, we define $t_i = i-2$.
    If~the boundary of~$f$ is connected, $t_i$ corresponds to the number of triangular faces the face~$f$ would contain in a plane triangulation of~$\Gamma_0$.
    Otherwise, the number~$t_i$ provides a lower bound on this number.
    As a plane triangulation of~$\Gamma_0$ consists of~$2n-4$ faces, we obtain $\sum_{i=3}^{\infty} t_i \cdot a_i \leq 2n-4$.
    By \cref{lem:2-col_1-plane_crossings_in_face}, we have
    \[
    \abs{X} \leq \sum_{i=3}^{\infty} \ccross(i) \cdot a_i = \sum_{i=3}^{\infty} \frac{1}{t_i} \cdot \ccross(i) \cdot  t_i \cdot a_i \leq \sum_{i=3} \frac{4}{3} \cdot t_i \cdot a_i \leq \frac{8}{3} \cdot (n-2) = 2.\overline{6}(n-2).
    \]
    where we used in the second step, that $\frac{1}{t_i} \ccross(i) \leq \frac{1}{t_8} \ccross(8) = \frac{4}{3}$ for all~$i \geq 3$.
    Indeed, for~$i=3$, we have $\frac{1}{t_i} \ccross(i) = 0$ and for~$i=4$, we obtain $\frac{1}{t_i} \ccross(i) \leq 1$.
    For~$i$ with $5 \leq i \leq 7$ we have $\frac{1}{t_i} \ccross(i) \leq \frac{i-1}{i-2} \leq \frac{4}{3}$, and for~$i \geq 8$ we obtain~$\frac{1}{t_i} \ccross(i) \leq \frac{i}{i-2} \leq \frac{4}{3}$.
\end{proof}
As every $2$-colorable $1$-plane drawing can be augmented to a filled one by only adding plane edges, \cref{thm:crossings_in_2-col_1-plane} follows.
A drawing that matches the upper bound given in \cref{thm:crossings_in_2-col_1-plane} can be obtained from a tiling of $8$-gons by inserting a $\FVIII$-configuration in each face, see \cref{fig:crossing_density_2-col_1-plane_lower-1}.
\begin{figure}
    \savebox{\imagebox}{\includegraphics[page=1]{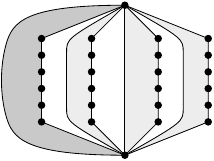}}
    \centering
    \subcaptionbox{\label{fig:crossing_density_2-col_1-plane_lower-1}}{\centering\usebox{\imagebox}}
    \hfil
    \subcaptionbox{\label{fig:crossing_density_2-col_1-plane_lower-2}}{\raisebox{\dimexpr.5\ht\imagebox-.5\height}{\includegraphics[page=2]{figures/crossing_density_2-col_1-plane_lower.pdf}}}
    \caption{Construction of a $2$-colorable $1$-plane drawing on $n$~vertices with $2.\overline{6}(n-2)$~crossings. (a) A plane drawing on $n = 26$~vertices where each face has size~$8$. (b) To each face~$f$ all diagonals connecting two vertices at distance~$2$ on the boundary of~$f$ are added.}
    \label{fig:crossing_density_2-col_1-plane_lower}
\end{figure}

\begin{proposition}
\label{prop:crossing_density_2-col_1-plane_lower}
    For every $n = 2+6\ell$ with $\ell \in \mathbb{N}^+$, there exists a non-homotopic $2$-colorable $1$-plane drawing of a graph on~$n$ vertices with $2.\overline{6}(n-2)$~crossings and $4(n-2)$~edges.
\end{proposition}

There are $3$-plane drawings matching both the maximum crossing and edge density among all $3$-plane drawings~\cite{goetze2025crossing}.
Yet, while the upper bounds for edge and crossing density for $2$-colorable $1$-plane drawings are tight (cf. \cref{prop:edge_density_2-col_1-plan_lower,prop:crossing_density_2-col_1-plane_lower}), no drawings matching both bounds are known.

\section{Relationship between \texorpdfstring{$\bm{t}$}{t} and \texorpdfstring{$\bm{k}$}{k}}
\label{sec:relationship_t_k}

Every $t$-colorable $k$-plane drawing is $tk$-plane, yet there are $tk$-plane drawings which are not $t$-colorable $k$-plane.
Indeed, $2$-plane drawings on $n$~vertices can contain up to $5(n-2)$~edges \cite[Theorem~1]{pach1997graphsFewCrossings}, while the upper bound for $2$-colorable $1$-plane drawings is~$4.\overline{3}(n-2)$  (cf. \cref{cor:edge_density_2-col_1-plane_upper}) and we also know that the densest $3$-plane drawings are not $3$-colorable $1$-plane (cf. \cref{sec:edge_density_3-col_1-plane}).
Yet, some of the densest $4$-plane drawings are also $2$-colorable $2$-plane.
Is in fact \emph{every} $4$-plane drawing also $2$-colorable $2$-plane?
This turns out not to be the case.
\begin{theorem}
\label{thm:tk-plane_not_t-col_k-plane}
    For every~$k \in \N$ and $t \geq 2$ there exists a $kt$-plane drawing~$\Gamma$ that is not $t$-colorable $k$-plane.
\end{theorem}
\begin{proof}
    Consider a drawing~$\Gamma$ on $kt+1$~pairwise intersecting edges.
    The corresponding conflict graph is the complete graph~$K_{kt+1}$.
    As $\Delta(K_{kt+1}) = kt$, each edge is crossed at most $kt$ times and the drawing $\Gamma$ is $kt$-plane.
    However, $\Gamma$ is not $t$-colorable $k$-plane as $\varphi_k(K_{kt+1}) \geq t+1$ (cf. \cref{obs:correspondence_frugal_conflict_graph}).
    Indeed, in every $t$-vertex-coloring of~$K_{kt+1}$ there exists a color class~$C$ of size at least~$\left\lceil{\frac{kt+1}{t}}\right\rceil = k+1$.
    Thus, every vertex~$v \in V(K_{kt+1})$ that is not part of~$C$ has at least $k+1$ monochromatic neighbors, i.e. $\Phi$ is not $k$-frugal.
\end{proof}

In general~$t$ and~$k$ are not interchangeable, that is there are $t$-colorable $k$-plane drawings that are not $k$-colorable $t$-plane.
Yet, are these drawings $f(k,t)$-colorable $t$-plane for some~$f(k,t)$ which only depends on~$k$ and~$t$?
The conflict graph~$H$ of a $t$-colorable $k$-plane drawing has maximum degree~$\Delta(H) \leq kt$.
As Kang and Müller bounded the $k$-frugal coloring number of a graph~$H$ in terms of~$\Delta(H)$, this immediately yields the following (cf. \cref{obs:frugal_lower_bound,obs:correspondence_frugal_conflict_graph}).
\begin{theorem}[{Kang and Müller~\cite[Theorem~3.5]{kangFrugalAcyclicStar2011}}]
\label{cor:t-col_k-plane_k-col_t-plane}
    There is a function~$f\colon \N \times \N \to \N$ such that every $t$-colorable $k$-plane drawing is $f(k,t)$-colorable $t$-plane for every~$t$ and~$k$.
\end{theorem}

\begin{figure}
        \savebox{\imagebox}{\includegraphics[page=9]{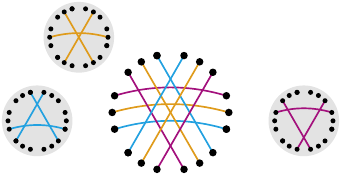}}
    \centering
    \begin{subfigure}[b]{0.4\textwidth}
        \centering
        \usebox{\imagebox}
        \subcaption{}
        \label{fig:t-k_k-t_drawing}
    \end{subfigure}
    \begin{subfigure}[b]{0.4\textwidth}
        \centering
        \raisebox{\dimexpr.5\ht\imagebox-.5\height}{\includegraphics[page=8]{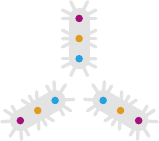}}
        \subcaption{}
        \label{fig:t-k_k-t_frugal_coloring}
    \end{subfigure}
    \caption{
    (a) A $3$-colorable $2$-plane drawing (center) on nine edges that is not $2$-colorable $3$-plane. The three color classes of the $2$-plane $3$-edge-coloring are represented separately in the gray regions. (b) The conflict graph of~$\Gamma$ with the corresponding $2$-frugal $3$-coloring.}
    \label{fig:t-k_k-t_t_construction}
\end{figure}

We have already seen examples of $t$-colorable $k$-plane drawings that are not $k$-colorable $t$-plane for small values of~$t$ and~$k$.
In the remainder of this section, we provide a general family of such examples and prove the following.
\begin{restatable}{proposition}{TKNotKTPlanar}
\label{prop:k_t_t_k_construction}
    For every $k,t \in \N$ such that $t \nequiv 0 \mod k$, there exists a $t$-colorable $k$-plane drawing that is not $k$-colorable $t$-plane.
\end{restatable}

To construct such drawings, we first need some drawing~$\Gamma$ where every $k$-plane edge-coloring uses exactly $t$~colors.
That is, for the conflict graph~$H$ of~$\Gamma$, we have $\varphi_k(H) = t$ (cf. \cref{obs:correspondence_frugal_conflict_graph}).
In fact, the \emph{$t$-blow-up}~$H$ of the complete graph~$K_{k+1}$ has this property (cf. \cref{prop:blow_up_properties}).
The graph~$H$ is the (balanced) complete $(k+1)$-partite graph on $(k+1)t$ vertices.
That is, $H$ consists of $k+1$ \emph{parts}~$A_1, \dots, A_{k+1}$, each containing $t$~vertices.
Vertices of different parts~$A_i$ are joined with edges, yet there are no edges within any part~$A_i$, i.e.,
    \[E(H) = \set{v_iv_j \given i,j \in [k+1], i \neq j, v_i \in A_i, v_j \in A_j}.\]

\begin{proposition}
\label{prop:blow_up_properties}
    Let $k,t \in \N$ and~$H$ be the $t$-blow-up of the complete graph~$K_{k+1}$ with parts~$A_1, \dots, A_{k+1}$, each containing $t$~vertices.
    We have
    \begin{enumerate}[(i)]
        \item\label{itm:k-plane_t-colorable} $\varphi_k(H) = t$
        \item\label{itm:coloring_all_distinct_colors} in every $k$-frugal $t$-coloring of~$H$, all vertices in every part~$A_i$ have distinct colors, and
        \item\label{itm:not_t-plane_k-colorable} if $t \nequiv 0 \mod k$, then $\varphi_t(H) > k$.
    \end{enumerate}
\end{proposition}
\begin{proof}
    To prove (\ref{itm:k-plane_t-colorable}), observe that coloring for every color~$c \in [t]$ exactly one vertex of each part~$A_i$ in~$c$, yields a $k$-frugal $t$-coloring, that is~$\varphi_k(H) \leq t$.
    See \cref{fig:t-k_k-t_frugal_coloring} for an illustration.
    As $\Delta(H) = kt$, $\varphi_k(H) = t$ follows with \cref{obs:frugal_lower_bound}.

    We now show $\varphi_t(H) > k$ if $t \nequiv 0 \mod k$, which corresponds to (\ref{itm:not_t-plane_k-colorable}).
    Suppose there exists a $t$-frugal $k$-coloring~$\Phi_t$ of~$H$.
    As all the vertices in $A \coloneqq A_1 \cup \dots \cup A_k$ lie in the neighborhood of each vertex in~$A_{k+1}$, each color appears exactly $t$ times on~$A$.
    In particular, there exists for every color~$c \in [k]$ a part~$B_c \coloneqq A_i$ with~$i \in [k]$ such that $\abs{A_i \cap \Phi^{-1}(c)} \leq \left\lfloor \frac{t}{k} \right\rfloor$.
    That is, color~$c$ appears at least $t-\left\lfloor \frac{t}{k} \right\rfloor$ times on~$N(v) \cap A$ for each vertex~$v \in B_c$.
    In particular, at most $\left\lfloor{\frac{t}{k}}\right\rfloor$ vertices of~$A_{k+1}$ can be colored in~$c$.
    As there are only $k$~colors in total, at most $k \cdot \left\lfloor{\frac{t}{k}}\right\rfloor$ of the $t$~vertices of~$A_{k+1}$ can be colored accordingly.
    Yet, if $t \nequiv 0 \mod k$, the number~$k$ does not divide~$t$, a contradiction.

    It remains to prove (\ref{itm:coloring_all_distinct_colors}).
    Consider a $k$-frugal $t$-coloring~$\Phi$ of~$H$.
    We need to show for all~$i$ that $\Phi(u_i) \neq \Phi(v_i)$ for all $u_i,v_i \in A_i$ with $u_i \neq v_i$ for each part~$A_i$.

    Suppose some color~$c$ appears twice on a part~$A_i$.
    We now show that color~$c$ does not appear on any other part~$A_j$.
    Consider a vertex~$v_j \in A_j$ for any~$j \neq i$.
    As~$\deg(v_j) = kt$, each color appears exactly $k$ times on~$N(v_j)$.
    In particular, we have $\abs{\Phi^{-1}(c) \cap N(v_j)} = k$.
    As color~$c$ appears twice on~$A_i$, there exists a part~$A_{\ell}$ with $\ell \neq j$ where color~$c$ does not appear.
    Now consider a vertex~$v_{\ell} \in A_{\ell}$.
    As $\deg(v_{\ell}) = kt$ and $\abs{\Phi^{-1}(c) \cap A_{\ell}} = 0$, we have
    \begin{align*}
    k &= \abs{\Phi^{-1}(c) \cap N(v_{\ell})} = \sum_{h \neq \ell} \abs{\Phi^{-1}(c) \cap A_h} =  \sum_{h \in [k+1] } \abs{\Phi^{-1}(c) \cap A_h}\\
    &= \abs{\Phi^{-1}(c) \cap A_j} + \sum_{h \neq j} \abs{\Phi^{-1}(c) \cap A_h} \geq \abs{\Phi^{-1}(c) \cap N(v_j)} = k,
    \end{align*}
    We thus obtain $\abs{\Phi^{-1}(c) \cap A_j} = 0$.

    Color~$c$ does not appear on any part~$A_j$ with~$j \neq i$, i.e., the neighborhood of each vertex~$v_i \in A_i$ uses $(t-1)$ colors.
    Yet, $\deg(v_i) = kt$.
    A contradiction to~$\Phi$ being $k$-frugal.
\end{proof}

There are drawings whose conflict graph corresponds to the $t$-blow-up of~$K_{k+1}$:
For each~$i \in [k+1]$, consider a matching represented by a drawing~$\Gamma$ of~$t$ parallel segments~$S_i$ in the plane such that each~$s_i \in S_i$ intersects all segments~$s_j \in S_j$ with $i \neq j$, see also \cref{fig:t-k_k-t_drawing}.
\begin{observation}
\label{obs:drawing_with_t-blowup-conflict-graph}
    For every~$k,t \in \N$, there exists a drawing~$\Gamma$ whose conflict graph is the $t$-blow-up of the complete graph~$K_{k+1}$.
\end{observation}
This completes the proof of \cref{prop:k_t_t_k_construction}.
We make use of such drawings to lift the $\NP$-completeness result for recognition of $t$-colorable $1$-plane drawings and $2$-colorable $2$-plane drawings (cf. \cref{thm:recognition_k_1} and \cref{thm:recognition_2_2}) to $t$-colorable $k$-plane drawings (cf. \cref{thm:recognition_collected}, see \cref{sec:app:t_to_t_plus_one} for details).
In particular Property~\eqref{itm:coloring_all_distinct_colors} of~\cref{prop:blow_up_properties} will come in handy in this proof.

\section{Recognizing \texorpdfstring{$\bm{t}$}{t}-colorable \texorpdfstring{$\bm{k}$}{k}-plane drawings}\label{sec:recognition}
Here we consider the recognition problem of $t$-colorable $k$-plane drawings, which we assume to be given as a planarization of the drawing.
For $t=1$, it suffices to verify that no edge is crossed more than $k$ times.
For $t=2$ and $k=1$, the problem still lies in $\P$.

\begin{theorem}\label{thm:recognition_2_1}
	Given a $2$-plane drawing~$\Gamma$ of a graph~$G$ on $n$~vertices and $m$~edges, we can compute in time~$\mathcal{O}(n+m)$ a $1$-plane $2$-coloring~$\varphi\colon E(G) \to [2]$, or conclude that none exists.
\end{theorem}
\begin{proof}
	The conflict graph of the $m$ edges of $\Gamma$ consists of cycles and paths.
    All vertices of a path colored with the pattern orange, orange, blue, blue have a blue and a orange neighbor.
	Cycles admit such a coloring if and only if they are of size $s \equiv 0 \mod 4$ (cf. \cref{obs:2-plane_inner_cell}).
	This can be checked and (if applicable), such a coloring can be computed in $O(n+m)$ time.
\end{proof}

All other settings turn out to be \NP-complete, even for straight-line drawings.
In the following sections we show the reduction for $3$-colorable $1$-plane drawings (\cref{sec:app:gadgets_3_1}) for which we first establish a more abstract framework problem (\cref{subsec:framework}).
The proof for the remaining settings can be found in the appendix.

\subsection{A framework problem as a base for the reductions}\label{subsec:framework}

\newcommand{\conn}{connector\xspace}
\newcommand{\conns}{connectors\xspace}
\newcommand{\Conn}{Connector\xspace}
\newcommand{\Conns}{Connectors\xspace}
\newcommand{\mathC}{\ensuremath{C}\xspace}
\newcommand{\mathc}{\ensuremath{c}\xspace}

To show that, for all remaining values of $t$ and $k$, the problem is \NP-complete, we will use reductions from an intermediate problem called \textsc{Consistent Segment $t$-coloring Problem} (\tcsc).
An instance $I$ of \tcsc is an integer $t$, a set of non-crossing segments $S$ and a set of \conns \mathC.
A \conn \emph{connects} a list of segments and imposes constraints on their relative colorings.
There are 5 types of \conns.

\begin{enumerate}
    \item The \emph{equality} \conn (connected segments: $s_1, s_2$) enforces $\varphi(s_1) = \varphi(s_2)$.
    \item The \emph{inequality} \conn (connected segments: $s_1, s_2$) enforces $\varphi(s_1) \not= \varphi(s_2)$.
    \item The \emph{splitter} \conn (connected segments: $s_1, s_2, s_3$) enforces $\varphi(s_1) = \varphi(s_2) = \varphi(s_3)$, i.e., it ``splits $s_1$ into $s_2$ and $s_3$''.
    \item The \emph{crossing} \conn (connected segments: $s_1, s_1', s_2, s_2'$) enforces $\varphi(s_1) = \varphi(s_2)$ and $\varphi(s_1') = \varphi(s_2')$. All other pairs, e.g., $\varphi(s_1)$ and $\varphi(s_1')$, are independent.
    We say that $s_1$ is connected to $s_2$ but neither to $s_1'$ nor $s_2'$ (intuitively the signals $s_1,s_2$ and $s_1',s_2'$ cross).
    \item The \emph{not-all-equal} \conn (connected segments: $s_1, s_2, s_3$ plus $s_i$ for $i\in\{4, \ldots, t\}$) prohibits $\varphi(s_1) = \varphi(s_2) = \varphi(s_3)$, i.e., three specific segments cannot have the same color. Any other segment $s_i$ (if present) with $i\in\{4, \ldots, t\}$ is required to have color $i$.
\end{enumerate}

Then $I$ is a yes-instance if there exists a $t$-coloring $\varphi\colon S \rightarrow [t]$ of the segments~$S$ consistent with the restrictions imposed by the \conns.

    \subparagraph*{Overview of proving \texorpdfstring{$\bm{\NP}$}{NP}-hardness for \texorpdfstring{\boldtcsc}{t-CSC}.}
    To show that \tcsc is $\NP$-hard, we reduce from the $\NP$-complete problem $\naetsat$ \cite[p.\,217]{schaefer1978complexity}.
    An instance of $\naetsat$ is a monotone 3SAT-formula $\phi$ in conjunctive normal form with variables~$V = \set{v_1, \dots, v_n}$ and clauses~$C$.
    Each clause consists of three positive literals of variables in $V$ (there are no negative literals).
    It is a yes-instance if and only if there is a truth-assignment of the variables, such that every clause contains at least one variable set to true and one to false.
    The \emph{incidence graph} of~$\phi$ is a graph on the vertex set $V \cup C$ with edges~$vc$ for every variable~$v \in V$ that is contained in a clause~$c \in C$.
    The incidence graph does not need to be planar.\footnote{
    In fact, while $\naetsat$ is $\NP$-hard, it is solvable in polynomial time if the incidence graph is planar~\cite{moret1988planar}.}

    We will use the incidence graph of the monotone 3SAT-formula $\phi$ of an instance of $\naetsat$ as the basis of our construction of a \tcsc instance.
    Specifically we will first present a number of adaptions of and augmentations to this graph and then replace most vertices with segments and most edges by \conns.
    The resulting $\tcsc$ instance admits a valid $t$-coloring of its segments if and only if the $\naetsat$-instance is satisfiable.
    The idea is the following: for each variable~$v_i$ of~$\phi$, there is a set of segments $T_i^{\lor}$ which are connected via splitter \conns and therefore are forced to have the same color.
    There are additional $t-2$ sets $T^w_1, \ldots, T^w_{t-2}$ of segments.
    Within one set the segments are again connected via splitter \conns and have the same color.
    With inequality connectors we ensure that no segments in two sets $T^w_i, T^w_j$ with $i\not=j$ have the same color.
    Further all segments in any $T_i^{\lor}$ are forced to have colors different from any segment in any $T^w_i$ (again using inequality connectors).
    Thus, segments of any $T_i^{\lor}$ can only be colored in one of the two remaining colors, say red or blue.
    If all segments~$T_i^{\lor}$ are blue, the variable~$v_i$ is set to true, if they are red, the variable is false.
    Each clause~$v_i \vee v_j \vee v_k$ is represented by a not-all-equal connector that connects segments of the different sets~$T_{\ell}^{\lor}$ for $\ell \in \set{i,j,k}$.
    We represent crossings in the drawing of the incidence graph with crossing connectors.

    For the details of how we adapt the incidence graph of the $\naetsat$-instance, then build a \tcsc instance based on the adapted graph (an example is also shown in \cref{fig:schematic_construction}) and for the proofs that the construction exhibits all required properties which lead up to the following theorem, we refer the reader to \cref{sec:app:framework_details}.

\begin{restatable}[\omitted]{theorem}{tcscHardness}
\label{thm:hardness_tcsc}
    The \tcsc problem is \NP-complete for any $t\geq 2$.
\end{restatable}

\begin{figure}[tbp]
    \centering
    \subcaptionbox{\label{fig:gadget_equality}}{\includegraphics[page=46]{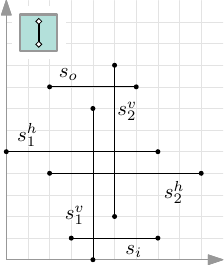}}
    \hfill
    \subcaptionbox{\label{fig:gadget_equality_colored}}{\includegraphics[page=47]{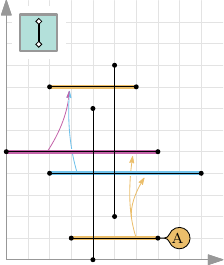}}
    \hfill
    \subcaptionbox{\label{fig:gadget_splitter}}{\includegraphics[page=52]{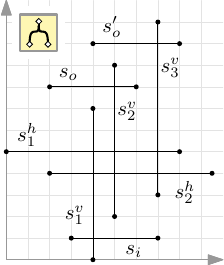}}
    \subcaptionbox{\label{fig:gadget_splitter_colored}}{\includegraphics[page=53]{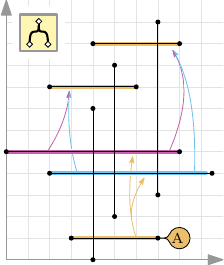}}
    \hfill
    \subcaptionbox{\label{fig:gadget_crossing}}{\includegraphics[page=49]{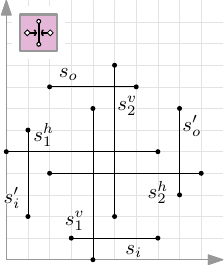}}
    \hfill
    \subcaptionbox{\label{fig:gadget_crossing_1}}{\includegraphics[page=50]{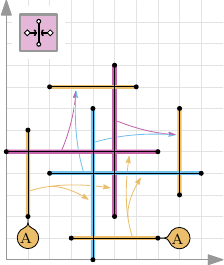}}
    \subcaptionbox{\label{fig:gadget_crossing_2}}{\includegraphics[page=51]{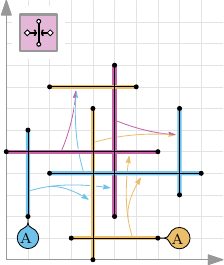}}
    \hfill
    \subcaptionbox{\label{fig:gadget_inequality}}{\includegraphics[page=58]{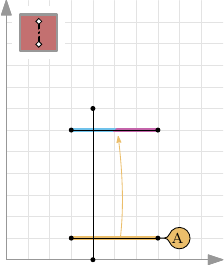}}
    \hfill
    \subcaptionbox{\label{fig:gadget_clause}}{\includegraphics[page=54]{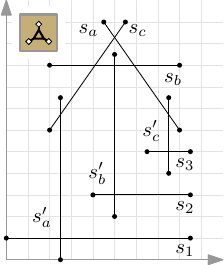}}
    \subcaptionbox{\label{fig:gadget_clause_unsat}}{\includegraphics[page=55]{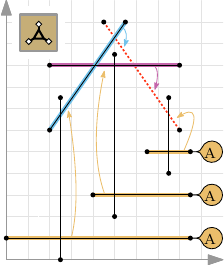}}
    \hfill
    \subcaptionbox{\label{fig:gadget_clause_2_1}}{\includegraphics[page=56]{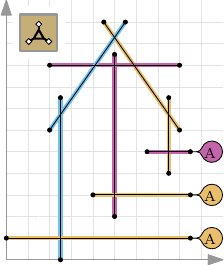}}
    \hfill
    \subcaptionbox{\label{fig:gadget_clause_1_1_1}}{\includegraphics[page=57]{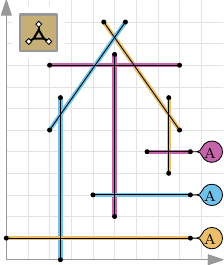}}
    \caption{$3$-colorable $1$-plane drawings for the equality (a)-(b), splitter (c)-(d), crossing (e)-(g), inequality (h), and not-all-equal (i)-(l) \conns, together with edge-colorings.
    A coloring marked with A is assumed.
    Arrows show when the coloring of one edge forces another one.
    Exact geometry is not crucial, only the intersection pattern.
    The red dashed edge in (j) cannot be colored validly.}
    \label{fig:gadgets}
\end{figure}

When reducing \tcsc to recognizing $t$-colorable $k$-plane drawings, segments are simply represented as edges in a drawing and it is sufficient to provide five drawings which mimic the coloring restrictions of the \conns in a $k$-plane $t$-edge-coloring.
We present the \conns for $t=3, k=1$ in the next section.

\subsection{Gadget correctness for the \texorpdfstring{$\bm{3}$}{3}-colorable \texorpdfstring{$\bm{1}$}{1}-plane setting}\label{sec:gadgets_3_1}

\newcommand{\gad}{gadget\xspace}
\newcommand{\gads}{gadgets\xspace}

In this section we present sets of segments, which contain special segments that -- if interpreted as the edges of a $3$-colorable $1$-plane drawing -- permit only very specific relative colorings.
These sets are called \gads.
Any \gad corresponds to a \conn~$C$ and contains a number of special segments which are identical to the segments that are connected by~$C$.

All gadgets are shown in \cref{fig:gadgets} and described in \cref{sec:app:gadgets_3_1} with a proof of their correctness.
Here, we only consider the equality \gad (\cref{fig:gadget_equality}) in more detail.
The equality \gad connects two segments $s_i$ and $s_o$.
It consists of two horizontal segments~$s^h_1, s^h_2$ and two vertical segments~$s^v_1, s^v_2$.
The two horizontal segments intersect the two vertical segments,
$s^v_1$ intersects $s_i$ and $s^v_2$ intersects $s_o$.
\begin{restatable}{lemma}{gadgetEquality}\label{lem:equality_3_1}
    The segments $s_i$ and $s_o$ have the same color in every $1$-plane $3$-coloring of the equality \gad.
\end{restatable}
\begin{proof}
    Let~$\varphi$ be a $1$-plane $3$-edge-coloring of the equality \gad.
    As $s_1^h$, $s_2^h$ and $s_i$ (respectively~$s_o$) all intersect $s_1^v$ (respectively~$s_2^v$), all three edges have distinct colors.
    That is, $\varphi(s_i), \varphi(s_o) \neq \varphi(s_j^h)$ for $j \in [2]$ and $\varphi(s_1^h) \neq \varphi(s_2^h)$.
    Hence, as only three colors are available $\varphi(s_i) = \varphi(s_o)$ follows.
\end{proof}

\begin{figure}[tbp]
    \savebox{\imagebox}{\includegraphics[page=48,scale=1.5]{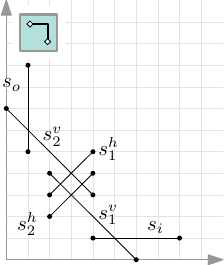}}
    \centering
    \begin{subfigure}[b]{0.4\textwidth}
        \centering
        \raisebox{\dimexpr.5\ht\imagebox-.5\height}{\includegraphics[page=46,scale=1.5]{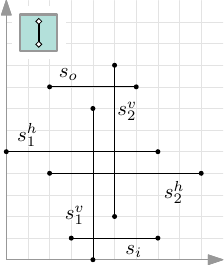}}
        \subcaption{}
        \label{fig:gadget_equality_restate}
    \end{subfigure}
    \qquad
    \begin{subfigure}[b]{0.4\textwidth}
        \centering
        \usebox{\imagebox}
        \subcaption{}
        \label{fig:gadget_corner}
    \end{subfigure}
    \caption{Two different drawings which both realize the equality \gad connecting~$s_i$, $s_o$.}
    \label{fig:gadget_corner_complete}
\end{figure}

The segments~$s_i$ and~$s_o$ need not both be horizontal, see \cref{fig:gadget_corner_complete}.
In fact, any drawing of the six segments~$s_i,s_o,s_1^h,s_2^h,s_1^v,s_2^v$ with the same intersection pattern as represented in \cref{fig:gadget_equality_restate} realizes the equality \gad.

Every \gad in isolation has a $1$-plane $3$-coloring, yet a segment connected to two \gads might pose unwanted coloring constraints.
Indeed, any segment can only be crossed once by each color.
This makes the colorings of these two \gads not entirely independent.
However, this obstruction can be circumvented:

\begin{lemma}\label{lem:independence_3_1}
    Let $g_1, g_2$ be two equality gadgets connecting segments $s_1$ to $s_2$ and $s_2$ to $s_3$, respectively. Let $s_a$ and $s_b$ be two segments intersecting $s_1$ and $s_2$, respectively.
    For any combination of $\phi(s_a)$ and $\phi(s_b)$ there is a valid $1$-planar $3$-coloring of $s_1, s_2, s_3$ as well as the segments in $g_1$ and $g_2$.
\end{lemma}
\begin{proof}
    Since $g_1, g_2$ are equality gadgets by \cref{lem:equality_3_1} we may assume that $\phi(s_1) = \phi(s_2) = \phi(s_3) = \text{yellow}$.
    It is sufficient to prove the following four cases: (i) $\phi(s_a) = \phi(s_b) = \phi(s_1)$, (ii) $\phi(s_a) \not = \phi(s_b) = \phi(s_1)$, (iii) $\phi(s_a) = \phi(s_b) \not= \phi(s_1)$ and (iv) $\phi(s_a) \not = \phi(s_b) \not= \phi(s_1)$.
    Colorings for all four cases are given in  \cref{fig:independence_3_1}.
\end{proof}

Replacing every segment that is connected to two \gads by three segments connected via two equality \gads (see \cref{fig:independence_3_1}) yields independence.
\begin{figure}
    \centering
    \begin{subfigure}[t]{.24\linewidth}
        \centering
        \includegraphics[page=59,scale=1.1]{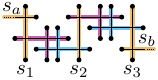}
        \subcaption{Case (i)}
    \end{subfigure}
    \hfill
    \begin{subfigure}[t]{.24\linewidth}
        \centering
        \includegraphics[page=60,scale=1.1]{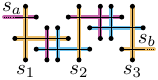}
        \subcaption{Case (ii)}
    \end{subfigure}
    \hfill
    \begin{subfigure}[t]{.24\linewidth}
        \centering
        \includegraphics[page=61,scale=1.1]{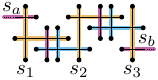}
        \subcaption{Case (iii)}
    \end{subfigure}
    \hfill
    \begin{subfigure}[t]{.24\linewidth}
        \centering
        \includegraphics[page=62,scale=1.1]{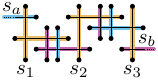}
        \subcaption{Case (iv)}
    \end{subfigure}
    \caption{Placing two equality \gads in a row makes the valid colorings of $s_a$ independent of the coloring of $s_b$.}
    \label{fig:independence_3_1}
\end{figure}
As the (now independent) gadgets mimic the behavior of the connectors, we obtain $\NP$-completeness.
\begin{theorem}\label{thm:recognition_3_1}
    Deciding whether a given straight-line drawing is $3$-colorable $1$-plane is \NP-complete.
\end{theorem}
\begin{proof}
    Clearly, it can be checked in polynomial time whether each edge of a drawing is crossed at most once by each color class of a given $3$-edge-coloring.
    We now prove $\NP$-hardness via a reduction from $\tcsc$ for $t=3$.
    Given an instance of \tcsc with $t=3$ we replace every gadget with the corresponding gadget represented in \cref{fig:gadgets} yielding a drawing $\Gamma$ (represented as a set of segments).
    An example of a \tcsc instance and the resulting drawing is shown in \cref{fig:3c1p_hardness_construction}.
    As each gadget behaves as required by the corresponding connector in every $1$-plane $3$-edge-coloring and the \gads can be colored independently, the \tcsc instance has a consistent coloring if and only if $\Gamma$ is $3$-colorable $1$-plane.
    This yields \NP-hardness.
\end{proof}

 \begin{figure}
        \centering
        \subcaptionbox{\label{fig:schematic_construction}}{\includegraphics[width=\linewidth,page=35]{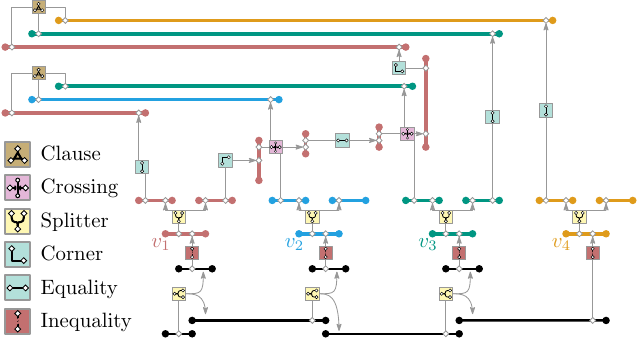}}
        \subcaptionbox{\label{fig:full_construction}}{\includegraphics[width=\linewidth,page=34]{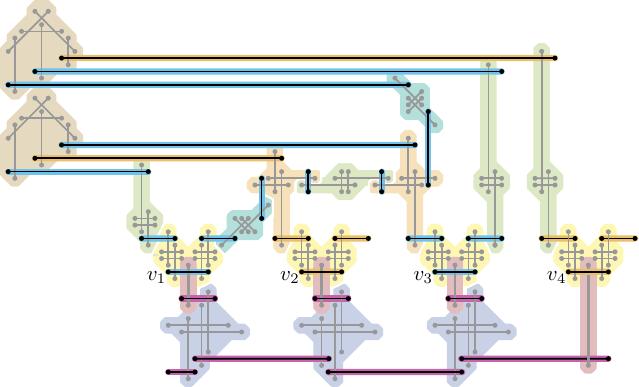}}
        \caption{(a) A \tcsc instance~$I$  with $\phi = (v_1 \lor v_2 \lor v_3) \land (v_1 \lor v_3 \lor v_4)$. The trees~$T_i^w$ are represented in black. Segments corresponding to the same variable of the $\naetsat$-instance have the same color. (b) The $3$-colorable $1$-plane drawing of~$I$. The coloring of the segments corresponds to a $1$-plane $3$-edge-coloring, the coloring of areas to the color-coding of the gadget's pictograms.
        }
        \label{fig:3c1p_hardness_construction}
    \end{figure}

\subsection{Extending the reduction to the remaining settings}
The gadgets represented in \cref{fig:gadgets} can be generalized to show the \NP-hardness of recognizing $t$-colorable $1$-plane drawings;  through a different set of gadgets we can prove a similar result for $2$-colorable $2$-plane drawings.
By providing a procedure that creates a drawing which is $t$-colorable $(k+1)$-plane if and only if a given base drawing is $t$-colorable $k$-plane, these results extend to all remaining cases.
In fact, this is similar to the $\NP$-completeness result of Bard, MacGillivray and Redlin  who characterize pairs~$t,k \in \N$ of integers for which recognizing graphs which admit \emph{proper} $k$-frugal $t$-vertex-colorings is $\NP$-hard \cite{bardComplexityFrugalColouring2021}.
They lift the $\NP$-completeness result of recognizing graphs with proper $k$-frugal $t$-colorings to proper $(k+1)$-frugal $t$-vertex-colorings: given a graph~$H$ they consider the graph~$H'$ obtained from~$H$ by adding for each vertex~$v \in V(H)$ a copy of~$K_{t}$ and identifying $v$ with a vertex of its copy~$K_{t}$.
They observe that~$H$ admits a proper $k$-frugal $t$-vertex-coloring if and only if $H'$ admits a proper $(k+1)$-frugal $t$-vertex-coloring.
In our context, $t$-blowups of the complete graph serve a similar role (cf. \cref{prop:blow_up_properties}).
The details can be found in \cref{sec:app:hardness}.

\begin{restatable}[\omitted]{theorem}{recognitionCollected}
\label{thm:recognition_collected}
    Deciding whether a given straight-line drawing is $t$-colorable $k$-plane is \NP-complete if $t=2, k\geq 2$ or $t\geq 3, k\geq 1$.
\end{restatable}

\section{Future Work}
\label{sec:conclusion}
Our research gives rise to structural questions with regard to the difference between fixed drawings and graphs, as well as complexity theoretic questions.

Every $t$-colorable $k$-plane drawing is in particular $tk$-plane.
However, there are $4$-plane drawings (cf. \cref{thm:tk-plane_not_t-col_k-plane}) and $3$-plane drawings (see \cref{fig:deg_3_no_2_frugal_2-coloring}) which are not $2$-colorable $2$-plane.
Yet, the corresponding graphs might admit different $3$- or $4$-plane drawings that are $2$-colorable $2$-plane.
This raises the following question.
\begin{question}\noindent
\label{question:k-planar_2-col_2-planar}
Is there a $4$- or $3$-planar graph which is not $2$-colorable $2$-planar?
\end{question}

While we constructed a $2$-colorable $1$-plane drawing of $n$-vertex graphs~$G$ that match the upper bound on the number of crossings, the graphs~$G$ might still admit a different $2$-colorable $1$-plane drawings with fewer crossings.
\begin{question}
    Are there $n$-vertex graphs where every $2$-colorable $1$-plane drawing contains $2.\overline{6}(n-2)$ crossings?
\end{question}

We only provided bounds on the number of crossings of $2$-colorable $1$-plane drawings.
\begin{question}
    What is the crossing density of $3$-colorable $1$-plane drawings?
\end{question}

In this work, we considered a colored variant of $k$-planarity.
Yet, the same questions on edge and crossing density, as well as recognition complexity can also be applied to different notions of beyond planarity, such as fan-planar or quasi-planar graphs.
In fact, for the uncolored variants many edge density results can be obtained using the density formula \cite{kaufmann2023density}.
Can the density formula be generalized to colored drawings?

While we considered in this work a coloring variant where different color classes interact (that is, each edge may only be crossed by up to $k$~edges of \emph{each} color), it would also be of interest to study a thickness variant: what is the minimum number~$t$ for a graph~$G$ such that~$G$ is the union of $t$ graphs, each of which is $k$-planar?

\bibliography{bib2doi}
\clearpage

\appendix

\section{Remaining hardness proofs for recognizing \texorpdfstring{\bm{$t$}}{t}-colorable \texorpdfstring{\bm{$k$}}{k}-plane drawings}\label{sec:app:hardness}

This section contains the missing proofs (\NP-completeness of \tcsc and correctness of the $3$-colorable $1$-plane gadgets) and extended results (\NP-completeness of all remaining settings) that were omitted from \cref{sec:recognition} due to space restrictions.

\subsection{Complete \texorpdfstring{\bm{\NP}}{NP}-completeness of \titletcsc}\label{sec:app:framework_details}

Here we give the ommitted detail of the hardness proof in \cref{subsec:framework}.

\textbf{Adapating the incidence graph.}
    The initial incidence graph and the adaptions described in this section are illustrated in Figure~\ref{fig:framework_adaptions}.
    We replace every variable vertex of a variable $v_i$ with two rooted binary trees $T_i^\lor$ (which has $t-2$ leaves) and $T_i^\land$ (has as many leaves as there are literals of $v_i$ in clauses).
    The two roots of $T_i^\lor$ and $T_i^\land$ are connected with an edge.
    Intuitively, the tree $T_i^\lor$ ensures that we can propagate the value of the variable $v_i$ to all of its literals, and the tree $T_i^\vee$ that no color other than red or blue is used on the segment corresponding to the variable $v_i$.

    We also add $t-2$ binary \emph{waste} trees (which might mean we add 0 of these trees in case $t=2$) each with $|V|+|C| + t - 2$ leaves.\footnote{Technically, one of these trees only requires $|V| + t - 2$ leaves, but this does not change the correctness of the reduction.}
    Each waste tree represents one color (other than red or blue) which may not be used on any of the segments corresponding to a variable~$v_i$.
    Later on, we need to ensure that only three colors are used on every clause: red, blue and one more color.
    To ensure all of the above, join leaves of the variable trees~$T_i^\vee$, leaves of waste trees and clause vertices such that
    \begin{itemize}
        \item each (former) leaf has degree at most~$2$,
        \item every pair of waste trees is joined by an edge,
        \item every variable tree~$T_i^\vee$ is joined with an edge to every waste tree,
        \item every clause is joined to all but one waste tree.
    \end{itemize}

    We can of course obtain a rectilinear drawing of the incidence graph in polynomial space, where all variable vertices are on a horizontal line at the bottom of the drawing and all clause vertices are on a vertical line on the left of the drawing.
    We now subdivide every edge of the drawing right before and after any crossing with another edge as well as just before and just after every bend.

    \begin{figure}
        \centering
        \includegraphics[page=8, width=\linewidth]{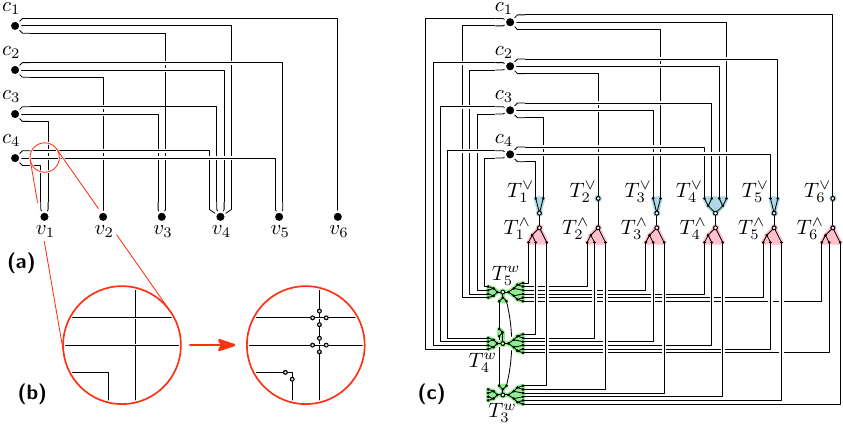}
        \caption{A rectilinear drawings of the initial incidence graph (a). Edges involved in crossings and bends (b) are subdivided. Variable vertices are replaced by two connected trees $T^\land$ and $T^\lor$ (c) and waste trees $T^w$ are added and connected to the clause vertices and the $T^\land$ trees. }
        \label{fig:framework_adaptions}
    \end{figure}

    \textbf{Construction of the \boldtcsc instance.}
    Next we replace every vertex except the clause vertices with a short line segment.
    If the vertex was a subdivision vertex on a horizontal part of an edge or a vertex of a waste tree we replace it with a short vertical segment, all other vertices are replaced by short horizontal segments.

    The segments are now connected with \conns based on their edge connections.
    In all binary trees a parent and its two children are connected by a splitter \conn.
    The roots of the two trees $T^\land_i$ and $T^\lor_i$ for every $i\in[|V|]$ are connected by an equality \conn.
    The edges incident to the leaves of the $T^\land$ trees, which are not part of $T^\land$ are replaced by inequality \conns.
    Similarly any edge between two leaves of different waste trees is replaced by an inequality \conn.
    Two crossing edges are replaced by a crossing \conn, s.t., the two segments corresponding to the two endpoints of the edge are forced to have the same color, respectively (we refer the reader to \cref{subsec:framework}).
    The clause vertices are replaced by not-all-equal \conns connecting all neighbors of the clause vertex, such that the three segments of the vertices on the paths to the $T^\lor$ trees are $s_1$, $s_2$ and $s_3$, while all remaining segments correspond to vertices on the path to a waste tree $T^w$.

    All remaining edges connect vertices of degree 2 and are replaced by equality \conns which might connect two horizontal, two vertical or a horizontal and a vertical segment.

    \begin{lemma}\label{lem:framework_one_color_in_tree}
        All segments within a single $T^w$, $T^\land$ or $T^\lor$ have the same color.
    \end{lemma}
    \begin{proof}
        Since all segments within a single waste tree are connected via splitter or equality \conns the lemma follows from the restrictions imposed by these \conns.
    \end{proof}

    Similarly we also state the following two observation.
    \begin{observation}\label{obs:framework_one_color_connected_to_tree}
        All segments which are connected to a $T^w$, $T^\land$ or $T^\lor$ tree via a chain of equality or crossing \conns also have the same color as the root of the tree they are connected to.
    \end{observation}
    \begin{observation}
        The two roots of a $T^\land$ and $T^\lor$ trees, which are connected via an equality \conn have the same color.
    \end{observation}

    Next we show that the waste trees cannot share the same color.
    \begin{lemma}\label{lem:framework_waste_trees_different}
        The roots of all waste trees have pairwise distinct colors.
    \end{lemma}
    \begin{proof}
        Since for every pair of waste trees, there is an inequality gadget connecting two segments which are connected via a chain of equality and crossing \conns to leaves of these trees, the leaves have different colors by \cref{obs:framework_one_color_connected_to_tree}.
        By \cref{lem:framework_one_color_in_tree} the roots of two waste trees have different colors.
    \end{proof}

    Now we can show how the waste trees restrict the colors which can be used in the $T^\land$ trees.

    \begin{lemma}\label{lem:framework_only_two_colors}
        The root of every $T^\land$ and $T^\lor$ tree uses one of two colors.
    \end{lemma}
    \begin{proof}
        Since for every pair of waste tree and $T^\land$ tree there is a connection between a leaf of the $T^\land$ tree and a segment connected to a leaf of the waste tree via a chain of equality or crossing segments, the root of the waste tree and the root of the $T^\land$ tree cannot have the same color.
        By \cref{lem:framework_waste_trees_different} the waste tree roots have $t-2$ pairwise unique colors and no root of any $T^\land$ or $T^\lor$ tree has any of these $t-2$ colors.
    \end{proof}

    \begin{figure}
        \centering
        \subcaptionbox{\label{fig:framwework_construction_detail}}{\includegraphics[page=11]{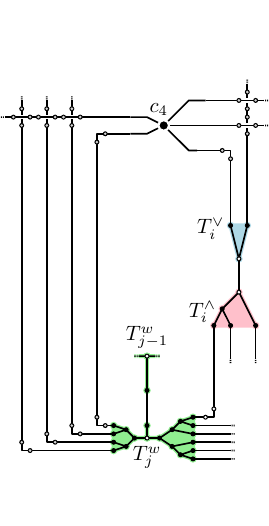}}
        \hfil
        \subcaptionbox{\label{fig:framwework_construction_segments}}{\includegraphics[page=10]{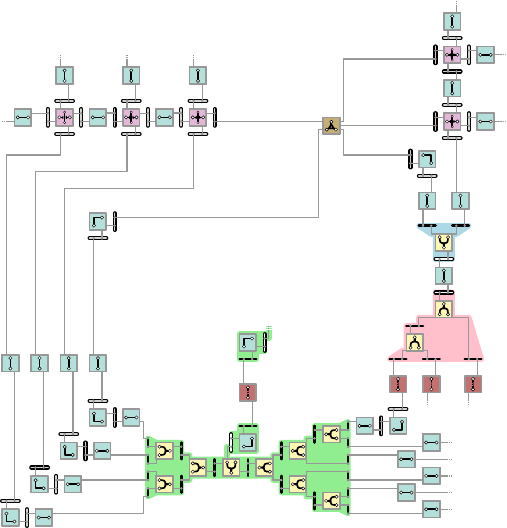}}
        \caption{The \tcsc instance is created based on the graph structure. A small part of the graph (a) in \cref{fig:framework_adaptions}c (for simplicity, this graph contains only two waste trees) and the corresponding section of the constructed instance (b) with matching color highlighting. Segments are black lines (partially with white or gray filling). \Conns are represented by small icons in blue (equality), red (inequality), yellow (splitter), purple (crossing) or brown (not-all-equal).}
        \label{fig:framwework_construction}
    \end{figure}

    In the following we will refer to the remaining two colors as the \emph{true color} (or simply true) and the \emph{false color} (or simply false).

    Finally we establish an observation, which will guarantee the second property required by the not-all-equal \conn.

    \begin{observation}\label{obs:framework_t_minus_three_unique_at_clause}
        Every not-all-equal \conn is connected (in addition to the segments $s_1, s_2$ and $s_3$) to $t-3$ segments, which are colored in $t-3$ pairwise unique colors, none of which are the true or the false color.
    \end{observation}
    \begin{proof}
        Since every not-all-equal \conn is connected to $t-3$ segments, which are connected via a chain of equality and crossing \conns to pairwise different waste trees, the lemma statement follows from \cref{lem:framework_waste_trees_different} and \cref{obs:framework_one_color_connected_to_tree}.
    \end{proof}

    We are now ready to prove \NP-hardness of \tcsc.

\tcscHardness*
\begin{proof}
    To prove hardness first assume that we are given a satisfying variable assignment for the given \naetsat instance.
    We color the roots of the trees $T^\land_i$ and $T^\lor_i$ and all segments connected via splitter and equality in the true/false color if $v_i$ is set to true/false in the satisfying variable assignment.
    If $t \geq 3$ we color the root segment of the waste tree which is not connected to any clause vertex (and all segments connected to them via splitter, equality or crossing \conns) in the third color.
    If $t \geq 4$ we color the roots of the remaining $t-3$ waste trees (and all segments connected to them via a chain of splitter, equality or crossing \conns) in the remaining $t-3$ colors.

    Then by construction the coloring of the segment is consistent with all equality, splitter and crossing gadgets.
    Since the waste tree roots have pairwise different colors, which are not the true or false color, the coloring is consistent with all inequality gadgets.
    Finally the three segments $s_x$, $s_y$ and $s_z$, which are forced to be colored differently by a not-all-equal gadget, are connected to three trees $T_x^\lor$, $T_y^\lor$ and $T_z^\lor$ and the \naetsat instance contains a clause $(v_x \lor v_y \lor v_z)$.
    Since not all variables have the same truth value in the satisfying assignment, $s_x$, $s_y$ and $s_z$ are not colored in the same color and the coloring is consistent with the not-all-equal gadget.

    Now assume that we are given a segment which is consistent with all gadgets.
    By \cref{lem:framework_only_two_colors} the roots of all $T^\lor$ trees are all colored in either the true or the false color.
    By a similar argument as in the previous paragraph for every clause in the \naetsat instance there are three segments $s_x$, $s_y$ and $s_z$, which are forced to be colored differently by a not-all-equal \conn and the root of the three $T^\lor$ trees these segments are connected to have the same color, respectively (again by \cref{lem:framework_only_two_colors}).
    We therefore simply set every variable $v_i$ to true/false if the root of $T_i^\lor$ is colored in the true/false color.

    Since the reduction creates at most a polynomial number of subdivision vertices, we can represent a solution to \tcsc in polynomial space as a simple list of colors, and we can check for every of the polynomially many gadgets in constant time if the coloring is consistent with the rules of the gadget.
    Then \NP-containment and consequently the theorem statement follows.
\end{proof}

\subsection{Correctness of the \texorpdfstring{\bm{$3$}}{3}-colorable \texorpdfstring{\bm{$1$}}{1}-plane gadgets}
\label{sec:app:gadgets_3_1}

Within this section we provide proofs for the correctness of the remaining gadgets represented in \cref{fig:gadgets}, namely the inequality, the splitter, the crossing and the not-all-equal \gads.

The simplest \gad is the inequality \gad (\cref{fig:gadget_inequality}) connecting two segments~$s_i$ and $s_o$.
It consists of a  single segment $s$ intersecting both two $s_i$ and $s_o$.
Thus, $s_i$ and~$s_o$ have distinct colors in every $1$-plane $3$-edge-coloring.
\begin{restatable}[\omitted]{lemma}{gadgetInequality}\label{lem:inequality_3_1}
    The segments $s_i$ and $s_o$ have different colors in every $1$-plane $3$-coloring of the inequality \gad.
\end{restatable}

Moreover we can easily extend the construction of the equality \gad by a third segment $s^v_3$, which intersects $s_1^h$ and $s_2^h$, and add a new segment $s_o'$, which intersects $s^v_3$ (\cref{fig:gadget_splitter}).
This construction yields a splitter \gad connecting $s_i, s_o$ and $s_o'$.
\begin{restatable}[\omitted]{lemma}{gadgetSplitter}\label{lem:splitter_3_1}
    The edges $s_i, s_o$ and $s_o'$ have the same color in every $1$-plane $3$-coloring of the splitter \gad.
\end{restatable}
\begin{proof}
For any pair of segments in $\{s_i, s_o, s_o'\}$ \cref{lem:equality_3_1} proves that they have the same color in any $1$-plane $3$-coloring.
\end{proof}

Based on the construction of the equality \gad (the segments $s^v_1$, $s^v_2$, $s^h_1$, and $s^h_2$) we can add four segments $s_i$, $s_o$, $s_i'$, and $s_o'$
intersecting $s^v_1$, $s^v_2$, $s^h_1$, and $s^h_2$, respectively.
This is a crossing \gad (\cref{fig:gadget_crossing_1} and \cref{fig:gadget_crossing_2}).

\begin{restatable}[\omitted]{lemma}{gadgetCrossing}\label{lem:crossing_3_1}
    The segments $s_i$ and $s_o$ have the same color in every $1$-plane $3$-coloring of the crossing \gad. The same holds for $s_i'$ and $s_o'$.
    The color of~$s_i$ and~$s_o$ is independent of the color of~$s_i'$ and~$s_o'$.
\end{restatable}
\begin{proof}
By \cref{lem:equality_3_1} the two segments in the pair $(s_i, s_o)$ (and similarly in the pair $(s_i', s_o')$) have the same color in any $1$-plane $3$-coloring.
This color can be the same for the two pairs (\cref{fig:gadget_crossing_1}) or different (\cref{fig:gadget_crossing_2}).
\end{proof}

The not-all-equal \gad is structurally quite different (\cref{fig:gadget_clause}).
Since $t=3$, the clause \conn of a \tcsc instance constructed as in the previous section is only connecting three segments $s_1, s_2$, and $s_3$.
The construction consists of three pairwise intersecting segments $s_a, s_b$ and $s_c$.
Three more segments $s_a', s_b'$, and $s_c'$ are added intersecting $s_a, s_b$, and $s_c$, respectively.
Additionally $s_a', s_b'$, and $s_c'$ intersect $s_1, s_2$, and $s_3$.

\begin{restatable}[\omitted]{lemma}{gadgetNAE}\label{lem:nae_3_1}
    There exists no $1$-plane $3$-coloring of the not-all-equal \gad which colors $s_1, s_2$ and $s_3$ with the same color.
    Moreover every other coloring of $s_1, s_2$ and $s_3$ can be extended to a $1$-plane $3$-coloring of the not-all-equal \gad.
\end{restatable}
\begin{proof}
        Note that since $s_a, s_b$ and $s_c$ pairwise intersect, if any two of them would be colored the same color, the third edge would intersect two edges of the same color.
        Therefore they all have a unique color.
        Next observe that $s_1$ and $s_a$ both intersect $s_a'$ and therefore have different colors.
        The same holds for $s_2$ and $s_b$, which intersect $s_b'$ as well as $s_3$ and $s_c$, which intersect $s_c'$.

        If $s_1, s_2$ and $s_3$ all have the same color, none of $s_a, s_b$ and $s_c$ can have that color, which is a contradiction, see \cref{fig:gadget_clause_unsat} for an illustration.

        Now assume that two segments of $s_1, s_2$ and $s_3$ have the same color with the third being colored differently,
        say $s_1$ and $s_2$ are orange and $s_3$ is purple (see \cref{fig:gadget_clause_2_1}).
        Now, we can color $s_a$ in orange, $s_b$ in purple and $s_c$ in blue.
        To obtain $1$-plane $3$-edge-coloring of the not-all-equal gadget, we can color $s_a'$ in blue, $s_b'$ in purple and $s_c'$ in orange.

        If all three segments of~$s_1,s_2,s_3$ have distinct colors the edge-coloring can be extended to a $1$-plane $3$-edge-coloring of the whole not-all-equal \gad as shown in \cref{fig:gadget_clause_1_1_1}.
\end{proof}

\subsection{Recognizing general \texorpdfstring{$\bm{t}$}{t}-colorable \texorpdfstring{$\bm{1}$}{1}-plane drawings is \texorpdfstring{\bm{\NP}}{NP}-hard}\label{sec:app:gadgets_k_1}
The reduction of the previous section is easily extended, yielding a reduction to the following result.
\begin{theorem}\label{thm:recognition_k_1}
    Deciding whether a given straight-line drawing is $t$-colorable $1$-plane is \NP-complete for $t \geq 3$.
\end{theorem}
\begin{proof}
The gadgets we need to construct for the reduction from \tcsc are simple adaptions of the gadgets presented in \cref{sec:app:gadgets_3_1}.
The adaptions we need to make are as follows.

The inequality gadget does not need to be generalized, as the base version of \cref{sec:app:gadgets_3_1} works for all values of $t$.
    \begin{figure}
        \centering
        \subcaptionbox{Equality gadget\label{fig:gadget_equality_t_1}}{\includegraphics[page=29]{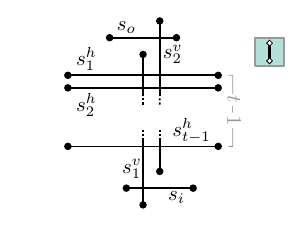}}
        \subcaptionbox{Splitter gadget\label{fig:gadget_splitter_t_1}}{\includegraphics[page=30]{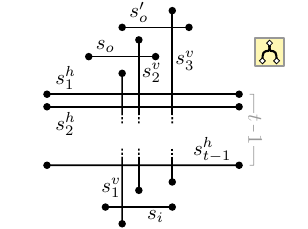}}
        \subcaptionbox{Crossing gadget\label{fig:gadget_crossing_t_1}}{\includegraphics[page=31]{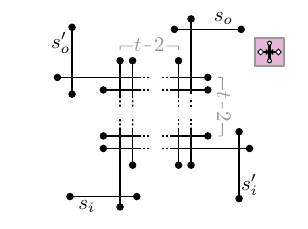}}
        \subcaptionbox{Not-all-equal gadget\label{fig:gadget_nae_t_1}}{\includegraphics[page=32]{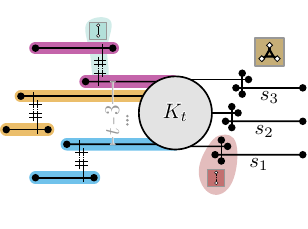}}
        \caption{The gadget constructions for the $t$-colorable $1$-plane setting. In the construction of the not-all-equal gadget (d) the pre-coloring of the fourth to $t$-th segment is indicated as well as the small copies of (in)equality gadgets are highlighted.}
        \label{fig:gadgets_t_1}
    \end{figure}

The equality gadget in \cref{fig:gadget_equality} for $t=3$ created two segments whose shared neighborhood had size $t-1 = 2$ and differed only in the two segments, which therefore were ensured to have the same color.
By increasing the number of horizontal segments, the segments shared neighborhood can be increased to general values of $t$ and as a result the construction in \cref{fig:gadget_equality_t_1} is an equality gadget for general $t$.
Of course the exact geometry is again irrelevant as long as the intersection pattern is identical and the gadget can be redrawn to obtain a corner gadget.

The generalization of the splitter gadget (\cref{fig:gadget_splitter_t_1}) works identically to the case of $t=3$ by simply adding an additional vertical line, creating three lines with a shared neighborhood of size $t-1$.

The same goes for the crossing gadget (\cref{fig:gadget_crossing_t_1}), where we place $t-1$ horizontal and $t-1$ vertical lines creating two segments whose neighborhood differs only in the segments $s_i$ and $s_o$ and a second pair of segments whose neighborhood differs only in  $s_i'$ ans $s_o'$.

The most involved construction is the general not-all-equal gadget (\cref{fig:gadget_nae_t_1}), since any value $t\geq4$ also involves additional segments with a fixed color (recall that for values $t\geq4$ the reduction in \cref{subsec:framework} connected $t-3$ waste tree segments, all with pairwise different colors, to the not-all-equal gadget).
The gadget consists of a set of $t$ pairwise intersecting segments, i.e., their conflict graph is the complete graph on $t$ vertices $K_t$ (in slight abuse of notation we will refer to these segments collectively as $K_t$).
The three input segments $s_1$, $s_2$, and $s_3$ are connected to three segments of $K_t$ with an inequality gadget, which we will call the \emph{clause segments} of the gadget.
The remaining $t-3$ edges (the \emph{waste segments} of the gadget) are each connected to one of $t-3$ input segments, which, by assumption of the gadget, have pairwise different colors $i\in\{4, \ldots, t\}$.
Therefore none of the clause segments of $K_t$ have a color in \{4, \ldots, t\}.
This leaves us with three pairwise intersecting segments, which can have one of three colors and which are connected via an inequality gadget to three input segments $s_1$, $s_2$, and $s_3$, which is the exact construction of the not-all-equal gadget of \cref{fig:gadget_clause} and it follows from \cref{lem:nae_3_1} that this construction is a generalized not-all-equal gadget.

The \gads can again be colored independently since \cref{lem:independence_3_1} immediately extents to a $t$-colorable $1$-plane setting.
Since \NP-containment is similarly straightforward than in \cref{sec:app:gadgets_3_1}, this concludes the proof.
\end{proof}

\subsection{Gadget correctness for the \texorpdfstring{$\bm{2}$}{2}-colorable \texorpdfstring{$\bm{2}$}{2}-plane setting}\label{sec:app:gadgets_2_2}
For the $2$-colorable $2$-plane setting we will first present a helper gadget which forces two edges to have necessarily the same color.
    This is the required property of the equality gadget, however in the construction of the helper gadget these segments require two intersections and we will present later on the actual equality gadget, which only requires one intersection per segment.

    \begin{figure}
        \centering
        \subcaptionbox{Helper gadget\label{fig:gadget_bottle_2_2_construction}}{\includegraphics[page=17]{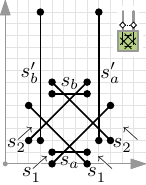}}
        \hfill
        \subcaptionbox{Valid coloring\label{fig:gadget_bottle_2_2_sat}}{\includegraphics[page=18]{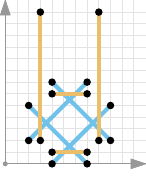}}
        \hfill
        \subcaptionbox{Invalid\label{fig:gadget_bottle_2_2_unsat_1}}{\includegraphics[page=19]{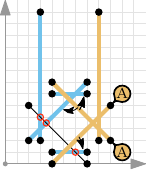}}
        \hfill
        \subcaptionbox{Invalid\label{fig:gadget_bottle_2_2_unsat_2}}{\includegraphics[page=20]{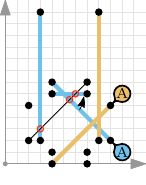}}
        \hfill
        \subcaptionbox{Invalid\label{fig:gadget_bottle_2_2_unsat_3}}{\includegraphics[page=21]{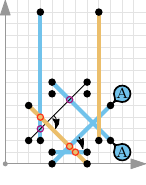}}
        \caption{The helper gadget for the $2$-colorable $2$-plane setting. The two extended segments are forced to have the same color; a valid coloring is shown in (b). Any coloring where the two segments $s_a'$ and $s_b'$ have different colors lead to a segment necessarily being crossed by three segments of the same color. Subfigures (c), (d), and (e) show the case distinction over the options for $s^\nearrow_1$ and $s^\nwarrow_2$. Assumed colorings of a case are indicated with an A, implications on the color of other segments are indicated with black arrows and purple circles, contradictions with red circles.}
        \label{fig:gadget_botlle_2_2}
    \end{figure}

    The helper gadget (\cref{fig:gadget_botlle_2_2}) consists of four inner diagonal segments $s^\nearrow_1$, $s^\nearrow_2$,$s^\nwarrow_1$, and $s^\nwarrow_2$, where $s^\nearrow_i$ and $s^\nwarrow_j$ intersect for any combination of $i,j\in\{1,2\}$.
    Additionally there are four outer segments: two horizontal segments $s_a$ (which intersects $s^\nearrow_1$ and $s^\nwarrow_1$) and $s_b$ (which intersects $s^\nearrow_2$ and $s^\nwarrow_2$) and two vertical segments $s_a'$ (which intersects $s^\nearrow_1$ and $s^\nwarrow_2$) and $s_b'$ (which intersects $s^\nearrow_2$ and $s^\nwarrow_1$).

    \begin{lemma}
    \label{lem:helper_gadget}
        The two segments $s_a'$ and $s_b'$ have the same color in any $2$-plane $2$-edge-coloring of the helper gadget.
    \end{lemma}
    \begin{proof}
        First observe that there is a valid coloring in which $s_a'$ and $s_b'$ have the same color, see \cref{fig:gadget_bottle_2_2_sat}.
        Now assume towards a contradiction that they have different colors and that $s_a'$ is orange.
        Then we distinguish three cases of the possible coloring of $s^\nearrow_1$ and $s^\nwarrow_2$.
        If both are orange (\cref{fig:gadget_bottle_2_2_unsat_1}) then $s^\nearrow_1$ and $s^\nwarrow_2$ both cross two orange segments.
        Any remaining segments including $s_a$ and $s^\nearrow_2$ are blue and $s^\nwarrow_1$ crosses three blue segments.
        If $s^\nearrow_1$ is orange and $s^\nwarrow_2$ is blue (\cref{fig:gadget_bottle_2_2_unsat_2}),  then $s^\nwarrow_2$ crosses two orange segments and $s_b$ is blue.
        Then $s^\nwarrow_2$ crosses three blue segments.
        The case where $s^\nearrow_1$ is blue and $s^\nwarrow_2$ is orange is symmetric.
        If $s^\nearrow_1$ and $s^\nwarrow_2$ are blue (\cref{fig:gadget_bottle_2_2_unsat_3}), then $s^\nearrow_2$ crosses two blue segments and $s^\nwarrow_1$ is orange.
        Now $s^\nearrow_1$ crosses two orange segments and $s_a$ is blue.
        As a result $s^\nwarrow_1$ crosses three blue segments.
        This concludes the proof.
    \end{proof}

    We will call $s_a'$ and $s_b'$ the \emph{connection segments} of the helper gadget.
    If both~$s_a'$ and~$s_b'$ of the helper gadget~$g$ intersect a segment~$s$ that is not in~$g$, we say that~$g$ is connected to~$s$.

    \begin{figure}
        \centering
        \subcaptionbox{Equality\label{fig:gagdet_equality_2_2}}{\includegraphics[page=29]{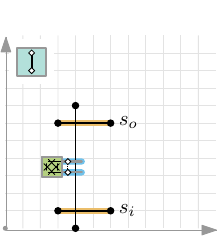}}
        \hfill
        \subcaptionbox{Splitter\label{fig:gagdet_splitter_2_2}}{\includegraphics[page=30]{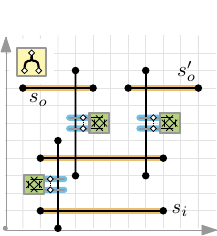}}
        \hfill
        \subcaptionbox{Inequality\label{fig:gagdet_inequality_2_2}}{\includegraphics[page=31]{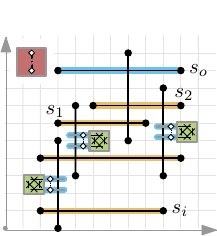}}
        \subcaptionbox{Crossing (same signal)\label{fig:gagdet_crossing_2_2_a}}{\includegraphics[page=32]{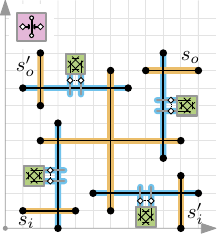}}
        \hfill
        \subcaptionbox{Crossing  (different signal)\label{fig:gagdet_crossing_2_2_b}}{\includegraphics[page=33]{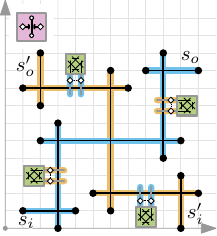}}
        \hfill
        \subcaptionbox{Not-all-equal\label{fig:gagdet_nae_2_2}}{\includegraphics[page=34]{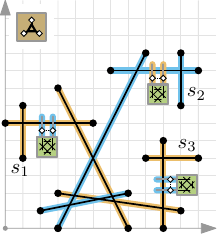}}
        \caption{The $2$-colorable $2$-plane gadgets. Helper gadgets are indicated by a green icon, whose protruding edges are the segments forced to be the same color. All gadgets are shown with a valid coloring (edges whose coloring can be freely chosen are left black).}
        \label{fig:gadgets_2_2}
    \end{figure}

    Now we proceed to present the five required gadgets starting with the actual equality gadget.
    An equality gadget (\cref{fig:gagdet_equality_2_2}) connecting two segments $s_i$ and $s_o$ is simply a segment $s$ which intersects both and is connected to a helper gadget.
    \begin{lemma}\label{lem:equality_2_2}
        Two segments $s_i$ and $s_o$ connected by an equality gadget have the same color in any $2$-plane $2$-edge-coloring.
    \end{lemma}
    \begin{proof}
        This lemma is immediate since $N(s) = \{s_i, s_o, s_c, s_c'\}$, where $s_c$ and $s_c'$ are the connection segments of the helper gadget which necessarily have the same color (cf. \cref{lem:helper_gadget}) and therefore the remaining two segments have the other color.
    \end{proof}
    Once again the geometry can be adapted to connect segments with the same or a differing orientation.
    Note that the segment $s$ can have any color.

    The splitter gadget (\cref{fig:gagdet_splitter_2_2}) connecting three segments $s_i, s_o$, and $s_o'$ consists simply of two equality gadgets connecting $s_i$ and $s_o$ as well as $s_i$ and $s_o'$.
    \begin{lemma}\label{lem:splitter_2_2}
        Two segments $s_i$ and $s_o$ connected by a splitter gadget have the same color in any $2$-plane $2$-edge-coloring.
    \end{lemma}
    \begin{proof}
        The statement follows from two applications of \cref{lem:equality_2_2}.
    \end{proof}

    The inequality gadget (\cref{fig:gagdet_inequality_2_2}) connecting the segments $s_i$ and $s_o$ consists of a splitter gadget connecting $s_i$ to two segments $s_1$ and $s_2$ and an additional segment $s_3$ which intersects $s_1$, $s_2$, and $s_o$.
    \begin{lemma}\label{lem:inequality_2_2}
        Two segments $s_i$ and $s_o$ connected by an inequality gadget cannot have the same color in any $2$-plane $2$-edge-coloring.
    \end{lemma}
    \begin{proof}
        By three applications of \cref{lem:splitter_2_2} $s_i$, $s_1$, and $s_2$ have the same color.
        Since $N(s_3) = \{s_1, s_2, s_o\}$, the segment $s_o$ has a color different from $s_1$, $s_2$ and therefore also different from the color of~$s_i$.
    \end{proof}

    The crossing gadget (\cref{fig:gagdet_crossing_2_2_a,fig:gagdet_crossing_2_2_b}) connecting $s_i$ to $s_o$ and $s_i'$ to $s_o'$ consists of two segments $s$ and $s'$, which cross each other.
    The segment $s$ is connected to $s_i$ and to $s_o$ with an equality gadget each.
    Similarly segment $s'$ is connected to $s_i'$ and to $s_o'$ with an equality gadget each.
    \begin{lemma}\label{lem:crossing_2_2}
        The segments $s_i$ and $s_o$ as well as the segments $s_i'$ and $s_o'$ which are connected by a crossing gadget have the same color in any $2$-plane $2$-edge-coloring.
        The color of~$s_i$ and~$s_o$ is independent of the color of~$s_i'$ and~$s_o'$.
    \end{lemma}
    \begin{proof}
        By \cref{lem:equality_2_2}, $s_i$ has the same color as $s$, which in turn has the same color as $s_o$.
        The same argument holds for $s_i'$, $s'$ and $s_o'$.
        As noted above, the central segment of the equality gadget, which is connected to the helper gadget can have any color, therefore the crossing gadget has a valid coloring if $s_i$ and $s_i'$ have the same color (\cref{fig:gagdet_crossing_2_2_a}) as well as when they have different colors (\cref{fig:gagdet_crossing_2_2_b}).
    \end{proof}

    Finally the not-all-equal gadget (\cref{fig:gagdet_nae_2_2}) connecting three input segments $s_1$, $s_2$, and $s_3$ consists of four pairwise intersecting \emph{clause} segments, three of which are connected with an equality gadget to $s_1$, $s_2$, and $s_3$, respectively.
    \begin{lemma}\label{lem:nae_2_2}
        The segments $s_1$, $s_2$, and $s_3$ which are connected by a not-all-equal gadget cannot all have the same color in any $2$-plane $2$-edge-coloring.
    \end{lemma}
    \begin{proof}
        Since the four clause segments pairwise intersect, no three of them can have the same color as the fourth would have three intersections with segments of the same color.
        However any combination of two segments being colored in orange and the other two blue is possible.
        Indeed, by \cref{lem:equality_2_2}, the three clause segments connected to $s_1$, $s_2$, and $s_3$ via equality gadgets have the color of the corresponding segment~$s_i$.
        Then the fourth clause segment can be colored such that there are exactly two clause segments of each color, if and only if $s_1$, $s_2$, and $s_3$ do not all have the same color.
    \end{proof}

\begin{theorem}\label{thm:recognition_2_2}
    Deciding whether a given drawing is $2$-colorable $2$-plane is \NP-complete.
\end{theorem}
\begin{proof}
    We proof this via reduction from $\tcsc$ for $t=2$.
    Given an instance of \tcsc with $t=2$ we replace every gadget with the corresponding construction of \cref{lem:equality_2_2,lem:inequality_2_2,lem:splitter_2_2,lem:crossing_2_2,lem:nae_2_2} yielding a drawing $\Gamma$ (represented as a set of segments).
    Then by the same lemmas the \tcsc instance has a consistent coloring if and only if $\Gamma$ is $2$-colorable $2$-plane and we have \NP-hardness.
    In contrast to the $t$-colorable $1$-plane setting, the coloring of the gadgets is immediately independent.
    This follows from the fact that any segment is connected to at most two \gads, is therefore only intersected by 2 segments and in a $2$-colorable $2$-plane setting, any combination of colors for these segments is allowed.

    Since an edge-coloring of $\Gamma$ can be represented as a list of integer values and it is straightforward to check in polynomial time if every edge in the drawing is crossed by every color at most twice, we also have \NP-containment, which concludes the proof.
\end{proof}

\subsection{Recognizing \texorpdfstring{$\bm{t}$}{t}-colorable \texorpdfstring{$\bm{k}$}{k}-plane drawings is \texorpdfstring{\bm{\NP}}{NP}-complete for \texorpdfstring{\bm{$k,t \geq 2$}}{k,t ≥ 2}}\label{sec:app:t_to_t_plus_one}
It is trivial that, given a planarization, recognition of $1$-colorable $t$-plane drawings is possible in polynomial time since it boils down to checking the number of intersections on every edge.
Previous sections have shown that we can recognize if a drawing is $2$-colorable $1$-plane in $\mathcal{O}(n+m)$ time (Theorem~\ref{thm:recognition_2_1}) and that recognition becomes \NP-hard in the $3$-colorable $1$-plane setting (Theorem~\ref{thm:recognition_3_1}), the general $t$-colorable $1$-plane setting (Theorem~\ref{thm:recognition_k_1})
as well as in the $2$-colorable $2$-plane setting (Theorem~\ref{thm:recognition_2_2}).

Here we show that all remaining settings are also $\NP$-complete by providing a construction that, given a drawing $\Gamma$ of a graph $G$, creates a drawing $\Gamma'$ of a graph $G'$ that is $t$-colorable $(k+1)$-plane but not $t$-colorable $k$-plane if and only if $\Gamma$ was $t$-colorable $k$-plane but not $t$-colorable $(k-1)$-plane.

The construction revolves around \emph{$(k+1)$-plane $t$-color generators}.
A \emph{$k$-plane $t$-color generator} is a (straight-line) drawing~$\Lambda$ that contains a set~$S$ of $t$ segments which are all incident to the outer face such that in every $(k+1)$-plane $t$-coloring of~$\Lambda$
\begin{enumerate}[(i)]
        \item\label{itm:dist_colors} all segments of~$S$ have distinct colors, and
        \item\label{itm:once_crossed} each segment of~$S$ is crossed exactly once.
\end{enumerate}
Intuitively, we intersect each edge of the given drawing~$\Gamma$ with such a set~$S$ to create the drawing~$\Gamma'$, thereby ``using up'' one crossing per color.
In fact, a drawing~$\Lambda'$ whose conflict graph is a $t$-blowup of the complete graph~$K_{k+2}$ already fulfills \eqref{itm:dist_colors} (cf. \cref{obs:drawing_with_t-blowup-conflict-graph}).
However, each segment of~$\Lambda'$ is already crossed~$t(k+1)$ times.
We thus need to make a small adaptation which requires a generalized equality gadget construction in order to obtain a $k$-plane $t$-color generator.

\begin{lemma}\label{lem:generalized_equality}
    Given two segments $s_1, s_2$ and two positive values $t$ and $k$ with $t>3$ or $t,k> 2$, we can add $tk+1$ segments in such a way that in every $k$-plane $t$-coloring $\phi$ of the segments, $\phi(s_1) = \phi(s_2)$.
\end{lemma}
\begin{proof}
    The construction is the natural extension of the equality gadget for $t$-colorable $1$-plane drawings described in the proof of \cref{thm:recognition_k_1}.
    We simply add one segment $s_a$ intersecting $s_1$, another segment $s_b$ intersecting $s_2$ and finally $kt-1$ pairwise non-crossing segments, which all intersect both $s_a$ and $s_b$.
    Since $N(s_a)- s_1 = N(s_b)- s_2$ and $\vert N(s_a)- s_1 \vert = \vert N(s_b)- s_2 \vert = kt-1$, both $s_a$ and $s_b$ are crossed $k$ times by $t-1$ colors and $k-1$ times by the last color.
    Therefore $s_1$ and $s_2$ must have that last color.
\end{proof}

Next we adapt the construction of the set of segments whose conflict graph is a $t$-blowup of $K_{k+1}$.

\begin{figure}
    \centering
    \begin{subfigure}{.35\linewidth}
        \centering
        \includegraphics[page=5]{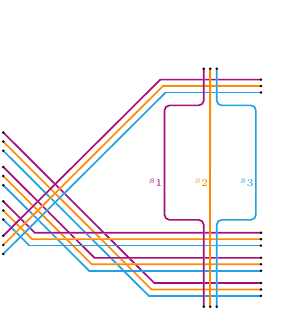}
        \subcaption[t]{}
        \label{fig:color_generator_a}
    \end{subfigure}
    \quad
    \begin{subfigure}{.16\linewidth}
        \centering
        \includegraphics[page=7]{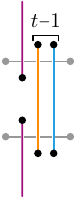}
        \subcaption[t]{}
        \label{fig:color_generator_b}
    \end{subfigure}
    \quad
    \begin{subfigure}{.35\linewidth}
        \centering
        \includegraphics[page=6]{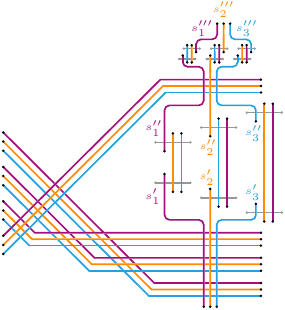}
        \subcaption[t]{}
        \label{fig:color_generator_c}
    \end{subfigure}
    \caption{
    The conflict graph of the segments in (a) is a $3$-blow up of $K_{4+1}$ (for clarity the segments are not drawn straight-line) where each segment~$s_1,s_2,s_3$ has a distinct color in every $4$-plane $3$-coloring.
    (b) We replace each of the vertical segments~$s_1,s_2,s_3$ with two new segments that are connected with an equality gadget.
    (c) Applying the replacement of~(b) twice to the vertical segments~$s_1,s_2,s_3$, we obtain three segments~$s_1''',s_2''',s_3'''$ of distinct colors, each of which is only crossed once and incident to the outer face.
    }
    \label{fig:placeholder}
\end{figure}

\begin{lemma}\label{lem:generator_low_degree}
    For all values $t$ and $k$, with $t>3$ or $t,k> 2$ there exists a (straight-line) $k$-plane $t$-color generator.
\end{lemma}
\begin{proof}
    Observe that, by \cref{obs:drawing_with_t-blowup-conflict-graph}, $k+1$ sets of $t$ pairwise non-crossing segments, which cross all other segments (see \cref{fig:color_generator_a}) already fulfill \eqref{itm:dist_colors}, however every segment contains exactly $tk$ crossings.
    To achieve \eqref{itm:once_crossed} we choose one set of $t$ segments.
    We replace every segment $s_i$ in this set by two shorter collinear segments $s'_i$ and $s''_i$, such that the disjoint union of the crossings on $s'_i$ and $s''_i$ is exactly the set of crossings on $s_i$ and both $s'_i$ and $s''_i$ contain at least one crossing.
    Then we connect $s'_i$ and $s''_i$ with the generalized equality gadget which, by \cref{lem:generalized_equality}, forces them to use the same color in any $k$-plane $t$-coloring (the replacement is represented in \cref{fig:color_generator_b}).
    Observe that $s'_i$ and $s''_i$ now both contain at least two crossings, at most $tk$ crossings and one of the two segments contains strictly less than $tk$ crossings.
    We may assume that $s''_i$ contains less than $tk$ crossings.
    We add one more segment $s_i'''$ and connect it with the generalized equality gadget to $s_i''$ (see \cref{fig:color_generator_c} for the whole construction).
    Now $s_i'''$ has the same color as $s'_i, s''_i$ and the original $s_i$.
    The set of all segments $s_i'''$ with $1\leq i \leq t$ contains $t$ segments, which have pairwise distinct colors and contain exactly one crossing (the one with the equality gadget) and property \eqref{itm:once_crossed} is fulfilled.
\end{proof}

With this we can prove the following result on \emph{$t$-colorable exactly $k$-plane} drawings (that is $t$-colorable $k$-plane drawings that are not $t$-colorable $(k-1)$-plane).
\begin{lemma}\label{lem:k_to_k_plus_one}
    Given a drawing $\Gamma$ that is $t$-colorable exactly $k$-plane we can construct a drawing~$\Gamma'$ that is $t$-colorable exactly $k+1$-plane in polynomial time in the size of~$\Gamma$.
    If~$\Gamma$ is straight-line, so is~$\Gamma'$.
\end{lemma}
\begin{proof}
    We simply place a copy of a $(k+1)$-plane $t$-color generator~$g(e)$ (as constructed in \cref{lem:generator_low_degree}) next to every edge $e$ of~$\Gamma$ and let~$e$ intersect the $t$~independent edges of $g(e)$, which contain only one crossing.
    By \cref{lem:generator_low_degree}, in every $(k+1)$-plane $t$-edge-coloring of the resulting drawing~$\Gamma'$ these $k+1$ all have distinct colors.
    The resulting drawing $\Gamma'$ is clearly $t$-colorable $(k+1)$-plane by \cref{lem:generator_low_degree}.
    Assume towards contradiction that $\Gamma'$ is still $t$-colorable $k$-plane.
    Then by removing the $(k+1)$-plane $t$-color generators, we would obtain a $(k-1)$-plane $t$-coloring of $\Gamma$ which is a contradiction.
\end{proof}

Finally this leads to the remaining settings being NP-complete.

\recognitionCollected*
\begin{proof}
    This is proven by induction, where $t=2, k=2$ (cf. \cref{thm:recognition_2_2}) as well as $t\geq 3, k=1$ (cf. \cref{thm:recognition_k_1}) are the (infinitely many) base cases.
    For the induction step, assume by hypothesis that for some fixed values $t = t'$ and $k=k'$ it is \NP-hard to recognize if a given drawing is $t'$-colorable $k'$-plane.
    Given a drawing as an instance of this problem, by using \cref{lem:k_to_k_plus_one} we create a drawing that is $t'$-colorable $(k'+1)$-plane if and only if the original drawing was $t'$-colorable $k'$-plane, yielding a reduction and proving that deciding if a drawing is $t'$-colorable $(k'+1)$-plane is also \NP-hard.

    Proving \NP-containment works the same as in previous sections.
\end{proof}

\end{document}